\documentclass{amsproc}
\usepackage{amssymb}

\usepackage{hyperref}

\newtheorem{theorem}{Theorem}[section]

\newtheorem{proposition}[theorem]{Proposition}

\theoremstyle{definition}

\newtheorem{example}[theorem]{Example}

\theoremstyle{remark}
\newtheorem{remark}[theorem]{Remark}

\numberwithin{equation}{section}

\newcommand{\dd}{\mathbf d}

\newcommand{\F}{\mathcal F}

\newcommand{\cH}{\mathcal H}
\renewcommand{\P}{\mathbb P}

\newcommand{\M}{\mathcal M}

\newcommand{\D}{\mathcal D}
\newcommand{\DD}{\mathbb D}
\newcommand{\N}{\mathbb N}

\newcommand{\I}{\mathcal I}

\newcommand{\A}{{\mathbb A}}

\newcommand{\cA}{\mathcal A}

\newcommand{\cR}{\mathcal R}

\newcommand{\cN}{\mathcal N}
\newcommand{\cC}{\mathcal C}
\newcommand{\cL}{\mathcal L}
\newcommand{\C}{\mathbb C}
\newcommand{\sC}{\mathcal C}
\newcommand{\Z}{\mathbb Z}

\renewcommand{\O}{\mathcal O}
\renewcommand{\L}{\mathbb L}

\newcommand{\Q}{\mathbb Q}

\newcommand{\wt}{{\rm wt}}

\newcommand{\half}{\frac{1}{2}}
\newcommand{\cyt}{Calabi--Yau threefold}
\newcommand{\cy}{Calabi--Yau}

\newcommand{\Hilb}{\mathop{\rm Hilb}\nolimits}
\newcommand{\IC}{\mathop{\rm IC}\nolimits}
\newcommand{\Tors}{\mathop{\rm Tors}\nolimits}
\newcommand{\NCHilb}{\mathop{\rm NCHilb}\nolimits}
\newcommand{\GL}{\mathop{\rm GL}\nolimits}
\newcommand{\SU}{\mathop{\rm SU}\nolimits}
\newcommand{\BGL}{\mathop{\rm BGL}\nolimits}
\newcommand{\SL}{\mathop{\rm SL}\nolimits}
\newcommand{\Tr}{\mathop{\rm Tr}\nolimits}
\newcommand{\Coh}{\mathop{\rm Coh}\nolimits}
\newcommand{\Pic}{\mathop{\rm Pic}\nolimits}
\newcommand{\Perv}{\mathop{\rm Perv}\nolimits}
\newcommand{\Const}{\mathop{\rm Const}\nolimits}
\newcommand{\Crit}{\mathop{\rm Crit}\nolimits}
\newcommand{\Hom}{\mathop{\rm Hom}\nolimits}
\newcommand{\Gr}{\mathop{\rm Gr}\nolimits}
\newcommand{\red}{\mathrm{red}}
\newcommand{\rat}{\mathrm{rat}}
\newcommand{\Mod}{\mathrm{mod}}
\newcommand{\MHM}{\mathop{\rm MHM}\nolimits}
\newcommand{\MMHM}{\mathop{\rm MMHM}\nolimits}
\newcommand{\MHS}{\mathop{\rm MHS}\nolimits}
\newcommand{\MMHS}{\mathop{\rm MMHS}\nolimits}
\newcommand{\Tot}{\mathrm{Tot}}
\newcommand{\Ext}{\mathrm{Ext}}
\newcommand{\vir}{\mathrm{vir}}

\begin{document}

\title[Cohomological Donaldson--Thomas theory]{Cohomological Donaldson--Thomas theory}

\author{Bal\'azs Szendr\H oi}
\address{Mathematical Institute, University of Oxford}
\email{szendroi@maths.ox.ac.uk}


\date{November 2014}

\begin{abstract} This review gives an introduction to cohomological Donald\-son--Thomas theory: the study of a 
cohomology theory on moduli spaces of sheaves on Calabi--Yau threefolds, and of complexes in 3-Calabi--Yau categories, 
categorifying their numerical DT invariant. Local and global aspects of the theory are both covered, including representations
of quivers with potential. We will discuss the construction of the DT
sheaf, a nontrivial topological coefficient system on such a
moduli space, along with some cohomology computations. The Cohomological
Hall Algebra, an algebra structure on cohomological DT
spaces, will also be introduced. The review closes with some recent appearances,
and extensions, of the cohomological DT story in the theory of knot invariants, of cluster algebras, and elsewhere. 
\end{abstract}    

\maketitle

\tableofcontents

\section*{Introduction}

Cohomological Donaldson--Thomas theory is the cohomological study of moduli spaces of sheaves on three-dimensional 
Calabi--Yau varieties. In its most developed form, it is a theory tied to a specific class
of geometries in a specific number of dimensions. It is hoped that the richness of the theory to be presented
will alleviate concerns over the specialized starting point. 

The aim of this review is to give a tutorial introduction to this subject, at a level accessible to graduate students 
wanting to get a glimpse and looking for pointers to the literature on specific topics. 
The discussion will start with an overview of {\em numerical} DT theory, effectively the Euler characteristic
shadow of the full theory, which was also discussed recently in the 
review~\cite{PT_13}, in particular its Sections $3\frac{1}{2}$ and  $4\frac{1}{2}$. While I~will not assume 
knowledge of~\cite{PT_13}, those hoping to learn about different facets of the subject 
should peruse both this review and~\cite{PT_13}. 

We will mostly work in the language of schemes, all schemes assumed to be of finite type over the base $\C$ of complex numbers. 
Algebraic stacks will make the occasional appearance, but most of our stacks will be global quotient stacks of the form $[M/G]$ with 
$G$ an affine algebraic group acting on a scheme $M$ without a geometric (scheme) quotient. It has recently become clear that 
the most natural context for the subject is that of derived symplectic geometry~\cite{PTVV}; because of (lack of) expertise, 
and the intention of keeping technicalities to a necessary minimum, I~will not give a full treatment of derived geometry, restricting
instead to the occasional side remark. The same applies to the main source of motivation for the historical 
development of this area of research: supersymmetric physics. 

\subsection*{Acknowledgements}

This review is a write-up of my lectures at the String Math Summer School held at the University of British Columbia
in June 2014. I~would like to thank its organizer Jim Bryan for the invitation to speak there. 
Thanks are also due to Kai Behrend, Tom Bridgeland, Jim Bryan, Ben Davison, Ian Grojnowski,
Dominic Joyce, Alastair King, Maxim Kontsevich, Davesh Maulik, Andrew Morrison,
Sergei Mozgovoy, Rahul Pandharipande, Yan Soibelman and
Richard Thomas for many valuable conversations on DT theory over the
years. Support from the Leverhulme Trust (Royal Society Leverhulme Trust Senior
Research Fellowship) and EPSRC (Programme Grant EP/I033343/1) during the
preparation of this review is also gratefully acknowledged.

\section{An overview of DT theory}
\label{sec:intro}

\subsection{Calabi--Yau threefolds}

We work over the field of complex numbers. 
A {\em Calabi--Yau threefold} is a quasiprojective three-dimensional non-singular 
algebraic variety $Y$ over $\C$, whose canonical (dualizing) line bundle $\omega_Y$ is trivial: $\omega_Y\cong \O_Y$. 
Sometimes simple connectedness of $Y$ or the weaker condition $H^1(\O_Y)=0$ will be assumed also, but we will not need 
this in general. As examples, it will be useful to keep the following list in mind, with varying amounts of projectivity. 
\begin{enumerate}
\item We can require $Y$ to be projective; there are hundreds of millions of families of examples, starting with the quintic
threefold $Y=Y_5\subset\P^4$. This will be the hardest case for the general theory to handle, so we will often study 
simpler geometries first.
\item Let $S$ be an algebraic surface, then $Y=\Tot(\omega_S)$, the total space of the canonical line bundle of $S$,
is a Calabi--Yau threefold. A special case is $S=\P^2$, with $Y=\Tot(\O_{\P^2}(-3))$.  
\item Let $C$ be a projective curve, and $\cL_1, \cL_2$ line bundles on $C$. Then the total space
$Y=\Tot(\cL_1\oplus\cL_2)$ is a Calabi--Yau threefold if and only if $\cL_1\otimes\cL_2\cong\omega_C$. A much-studied
special case is the resolved conifold geometry, when $C=\P^1$ and $\cL_i=\O_{\P^1}(-1)$. We will denote this space
by $Y=\O_{\P^1}(-1,-1)$. 
\item The simplest Calabi--Yau threefold of all, by our definition, is $Y=\C^3$.
\item A more general construction, giving some of the examples from (2)-(4) and others, is the following. Let $\Gamma<\SL(3,\C)$ 
be a finite group, and $\bar Y=\C^3/\Gamma$ the quotient, a singular variety with Gorenstein singularities 
and $\omega_{\bar Y}\cong\O_{\bar Y}$. Then there exist, usually several, quasi-projective Calabi--Yau resolutions 
$Y_i\to \bar Y$. A further generalization is to consider resolutions of 
other varieties $\bar Y$ with Gorenstein singularities and trivial canonical bundle. 
\end{enumerate}

\subsection{Sheaves on Calabi--Yau threefolds and DT invariants}

Let us for the time being restrict to projective Calabi--Yau threefolds $Y$ satisfying the extra condition
$H^1(\O_Y)=0$. One set of traditional, topological
invariants attached to such a geometry $Y$ are the topological Betti numbers $b_i(Y)=\dim H^i(Y, \Q)$ and the Hodge
numbers $h^{p,q}(Y)=\dim H^q(Y, \Omega_Y^p)$, the dimension of the coherent cohomology of $Y$ with coefficients
in the sheaf of holomorphic $p$-forms. Under the assumption $H^1(\O_Y)=0$, these two sets of numbers determine each other, 
so only depend on the underlying differentiable manifold. 
The initial aim of Donaldson--Thomas theory~\cite{T, DT} was to define more general
numerical invariants of a Calabi--Yau threefold $Y$, which are sensitive to its complex structure, while still being
invariant in the sense that they are unchanged under small deformations of the complex structure. These invariants
are attached to various moduli spaces of sheaves $\M=\M(Y)$ on $Y$.

To be slightly more precise, let us fix a cohomology class $\alpha\in H^{\rm even}(Y, \Q)$. Then we can consider the moduli {\em stack}
$\M(Y, \alpha)$ of coherent sheaves on~$Y$ with Chern character~$\alpha$. This is a non-Hausdorff Artin stack, usually of infinite
type, and as such does not admit sensible numerical invariants. Imposing a {\em stability condition} $\sigma$ may, in an ideal 
situation, give a proper, finite type and Hausdorff open substack $\M^\sigma(Y, \alpha)\subset \M(Y, \alpha)$ 
parametrizing $\sigma$-stable sheaves, with a {\em coarse moduli scheme} which by abuse of notation I~will 
also denote  $\M^\sigma(Y, \alpha)$. (Let me be vague for now on what sort of stability conditions one might impose.)
This moduli scheme of $\sigma$-stable sheaves 
is a better candidate for defining numerical invariants, but there is still a serious issue: in all but a handful 
of cases, $\M^\sigma(Y, \alpha)$ has very bad singularities. Its tangent space at a point $[E]\in \M^\sigma(Y, \alpha)$ 
corresponding to a
coherent sheaf $E$ on $Y$ is $T_{[E]} \M^\sigma(Y, \alpha) \cong \Ext^1(E,E)$, but there is a complicated higher-order map, the 
Kuranishi map, from this tangent space to the {\em obstruction space} $\Ext^2(E,E)$, whose zero set gives a local model for 
$\M^\sigma(Y, \alpha)$ near its point~$[E]$. 

The special feature for Calabi--Yau threefolds is that the tangent and obstruction spaces are naturally dual using
Serre duality and $\omega_Y\cong\O_Y$: 
\[\Ext^2(E,E) \cong \Ext^1(E,E\otimes\omega_Y)^\vee \cong\Ext^1(E,E)^\vee.
\] 
This observation was the starting point DT theory, formalized in the result~\cite{T} that
$\M^\sigma(Y, \alpha)$ carries a {\em perfect obstruction theory}~\cite{BF1} of dimension $0$, in fact a symmetric obstruction theory~\cite{BF}. 
A dimension $0$ perfect obstruction theory gives rise to a virtual fundamental class~\cite[Appendix]{PT_13}
\[ [\M^\sigma(Y, \alpha)]^\vir \in H_0(\M^\sigma(Y, \alpha), \Z),
\]
whose degree 
\[ \deg[(\M^\sigma(Y, \alpha))]^\vir\in \Z
\]
is the sought-for numerical invariant, the {\em DT invariant} of $Y$ in the class $\alpha$ and stability $\sigma$. 
It is important to understand that this is a {\em global} construction of the invariant, 
requiring $\M^\sigma(Y, \alpha)$ to be proper. 

It was proved by Thomas~\cite{T} that this quantity is indeed an invariant of the underlying threefold in the sense
explained above: if we fix a smooth family $\{Y_t\}$ of projective Calabi--Yau threefolds, a cohomology class $\alpha$ on the 
underlying smooth manifold, and a stability condition $\sigma$, the resulting number is constant in the family. 

The right stability condition to choose depends on the specifics of the situation under study. One option involves 
choosing some standard stability condition for sheaves such as Gieseker stability with respect to an ample line bundle. 
The numerical DT invariants change under change of stability condition following precise {\em wall-crossing relations},
as has been established as the outcome of extensive research whose highlights include~\cite{J1, JS, KS1}.

\begin{example} The equality of the dimensions of the tangent and obstruction spaces of the moduli scheme $\M^\sigma(Y, \alpha)$
at all its points says that its {\em virtual dimension} is $0$. This suggests that in sufficiently general,
transversal situations, $\M^\sigma(Y, \alpha)$ might consist of a finite number of points, perhaps even reduced points. 
In the latter case, and only in that case, the DT invariant is simply the number of points. While in algebraic
geometry such transversality is extremely rare, this indeed being the point of introducing the virtual 
fundamental class, there are some interesting cases, discussed
in~\cite[Thm. 3.55 and Sect. 4]{T}, where generically the moduli scheme is a 
finite set of reduced points. See Example~\ref{ex:ruledsurface:again}
below for another example. On the other hand, as soon as the DT
invariant is negative, such as in the explicit examples in Section~\ref{sec:macmahon} and
Example~\ref{ex:ruledsurface}, this cannot happen. 
\end{example} 

\begin{example} \label{ex_rank1}
For the case of sheaves of rank $1$, one can be more explicit about
the stability condition involved.  
Given a torsion-free sheaf $E$ of rank $1$ on $Y$, there is an injective map of sheaves $E\to \det(E)$ to the determinant
of $E$, which must be an invertible sheaf. Up to tensoring by this invertible sheaf, we can assume that the determinant is
trivial, and then we get an embedding $E\to\O_Y$ with cokernel the structure sheaf $Z\subset Y$ of a subscheme of $Y$, necessarily
supported in codimension two. The parameter space of such embedded subschemes is the {\em Hilbert scheme} of 
codimension-two subschemes of $Y$; under the assumption $H^1(\O_Y)=0$, the moduli scheme of sheaves is isomorphic to the Hilbert scheme and
is in particular separated and of finite type (for fixed topological invariants). See~\cite[Section $3\half$]{PT_13} for further
discussion of this important special case, which plays a major role in the historical development of the subject, 
as well as being responsible for important connections to other counting problems. 
\end{example}

\begin{remark} If the moduli problem of sheaves of class $\alpha$ involves $\sigma$-semistable objects, then the situation
changes radically. In this case, $\M^\sigma(Y, \alpha)$ is only a coarse moduli space, and its tangent space (as a scheme) can no longer be 
identified simply as an $\Ext$-group. In this case, DT invariants have to be defined by a complicated roundabout procedure~\cite{JS}, and 
they take values in $\Q$. We refer for more details about the issues, and their solution, to~\cite{JS}.
\end{remark} 

\subsection{Local interpretation: the Behrend function}

The above definition of the DT invariant was global: it needed a proper moduli space $\M=\M^\sigma(Y, \alpha)$, 
so that the degree of its virtual fundamental class 
is defined. In the case of invariants attached to symmetric obstruction theories, 
Behrend~\cite{Be} gave a radically different, local interpretation. What this work showed was that given any scheme $X$, 
associated to $X$ is a canonical {\em constructible} function $\nu_X\colon X\to \Z$,
defined in terms of local data, at a nonsingular point of $X$ its value being $(-1)^{\dim X}$. 
We can then define the {\em virtual Euler characteristic} of $X$ to be the integer
\[\chi_\vir(X) = \int_{X} \nu_X d\chi\in \Z.
\]
The latter integral is the integral of the constructible function $\nu_X$ against the (usual, topological) Euler characteristic, 
defined by the formula
\[\int_{X} \nu_X d\chi = \sum_{n\in \Z} n\cdot \chi\left(\nu_X^{-1}(n)\right).\]

We then have the following important result.

\begin{theorem} \cite{Be} Let $\M$ be a proper scheme equipped with a symmetric perfect obstruction theory, such as a proper moduli space
of stable sheaves on a projective Calabi--Yau threefold, so that $\M$ has a virtual fundamental class $[\M]^\vir$. Then
\[ \deg[\M]^\vir  = \chi_\vir(\M).
\]
\end{theorem}

The importance of this result is twofold. First of all, even if $\M$ is proper, it allows for stratification (``cut and paste'')
arguments since the right hand side (but not the left) behaves well under stratifications. In particular, the right hand side
is defined for any scheme, whether it is proper or not. This is the second advantage: the right hand side of this formula defines
DT invariants of non-compact Calabi--Yau threefolds as well. For such threefolds, we can still often define moduli spaces 
$\M=\M^\sigma(Y, \alpha)$, so long as we suitably restrict the set of objects under study. For example, we could ask for
moduli of sheaves which themselves have compact support inside $Y$. We then obtain 
finite type moduli schemes, but they are frequently non-compact. 
An extended example will be discussed in Section~\ref{sec:hilbC3}. 

\subsection{Invariants from other categories with $3$-Calabi--Yau properties}
\label{sec:3CYalg} 

The discussion above used moduli schemes, and defined numerical invariants, starting from the category $\sC=\Coh(Y)$ of coherent 
sheaves on a Calabi--Yau threefold~$Y$. We relied on three facts. The first, geometric
fact was that objects $E\in \sC$ have a well-defined moduli theory. The second, homological fact was Serre duality with trivial
dualizing sheaf, leading to the duality
of the tangent space $\Ext^1(E,E)$ and the obstruction space $\Ext^2(E,E)$. A final important requirement was that 
the category $\sC$ admits suitable notions of stability. 
One can hope to repeat the discussion for other categories $\cC$ with these properties.
We leave the most general definition to other sources~\cite{KS1}; we restrict to examples. 

\begin{enumerate}
\item For a Calabi--Yau threefold $Y$, we can consider $\sC=\D^b(\Coh Y)$, the {\em derived category} of coherent sheaves
on~$Y$. Serre duality works in exactly the same way as for sheaves. The most general notion of stability in this context is 
(Douglas-)Bridgeland stability~\cite{B}. 
\item Let $\cA$ be a finitely generated associative, not necessarily commutative algebra over $\C$. Let $\sC=\cA-\Mod$ be 
the category of finite dimensional right $\cA$-modules. Then the homological condition is satisfied if $\cA$ is 
a {\em 3-Calabi--Yau algebra}~\cite{G}. This means in particular that
\begin{enumerate}
\item for $M,N\in\cC$, $\Ext^i_\cA(M,N)$ is finite dimensional, and vanishes unless $0\leq i \leq 3$; 
\item for $M,N\in\cC$, there are perfect bifunctorial pairings
\[ \Ext^i_\cA(M,N)\times \Ext^{3-i}_\cA(N,M)\to\C.
\]
\end{enumerate}
As a subcase, one can consider $3$-Calabi--Yau algebras arising from {\em quivers with potential} (see Section~\ref{sec:quiv}); 
the objects in $\cA$ then readily admit a moduli theory, and (King) stability conditions~\cite{K}.
\end{enumerate}
DT invariants in these situations have been defined in~\cite{KS1}, respectively~\cite{Sz1, MR, KS2}.

\begin{example} An important special case of (1) is the stable pairs theory of Pandharipande and Thomas~\cite{PT}. 
This is a variant of Example~\ref{ex_rank1}, where we parametrize objects of $\D^b(\Coh Y)$ which are generically
ideal sheaves of embedded codimension-two subschemes $Z\subset X$ (curves), but with a stability condition different 
from considering torsion-free sheaves. A detailed discussion of this story can be found in~\cite[Sect. $4\half$]{PT_13}.
\end{example}

\begin{example} Interesting examples of (2) come from derived equivalences of the form 
\[ \D_c^b(\Coh Y) \cong \D^b(\cA-\Mod),
\]
for $Y$ a (noncompact) Calabi--Yau threefold, $\D_c^b(\Coh Y)$ the derived category of coherent sheaves with compact support 
on $Y$ and $\cA$ a finitely-generated commutative, usually non-commutative algebra. Such examples arise in the context
of the McKay correspondence~\cite{BKR}, and more generally for resolutions of three-dimensional 
Gorenstein singularities~\cite{vdB}. 
\end{example}

\subsection{Generating series of DT
  invariants and BPS invariants} \label{sec:genseriesDT} More than on individual invariants, the emphasis
of the subject has been on the study of generating series of the DT invariants. One reason for this is the origin
of these investigations in string theory, where it is natural to form generating series (or partition functions) of 
individual invariants, and study these as functions of the variables. In mathematical terms, alongside fixing the Calabi--Yau 
geometry $Y$ and a stability parameter $\sigma$, we also fix an affine sublattice $\Lambda\subset H^{\rm even}(Y, \Q)$ of rank $n$. 
We then consider an expression of the form
\[ Z_{Y, \Lambda}({\mathbf t}) = \sum_{\alpha\in\Lambda} \chi_\vir\left(\M^\sigma(Y, \alpha)\right) {\mathbf t}^\alpha\in \Z[[{\mathbf t}]].
\] 
Here ${\mathbf t}=\{t_1, \ldots, t_n\}$ is a set of auxiliary variables, and ${\mathbf t}^\alpha$ is multi-index notation. 

It was conjectured by~\cite{KS1, JS}, based on earlier calculations and physical considerations~\cite{GV}, 
that the generating series
of DT invariants $Z_{Y, \Lambda}({\mathbf t})$ take the form of certain infinite products with integral exponents 
$\Omega^\sigma(\alpha)\in \Z$. These invariants, sometimes called {\em
  BPS invariants}, are the fundamental structural quantities
governing DT generating series and their wall-crossing. The name comes from their interpretation as the number of BPS states
(states saturating a bound on their mass) in certain supersymmetric theories. The integrality of $\Omega^\sigma(\alpha)$ remains 
conjectural in general. See~\cite[Section $2\half$]{PT} for the parallel story in Gromov--Witten theory. See also
Section~\ref{subsec_COHA_pot} below.

\subsection{Categorifying DT invariants}\label{sec:catDT}
Given a moduli space of stable sheaves (or objects) $\M$ on a Calabi--Yau threefold (or CY3 category), 
the theory explained so far associates an integer invariant  
$\chi_\vir(\M)$ to $\M$, which, as explained above, can be expressed as a weighted Euler characteristic. The next question one can
ask is whether this construction can be {\em categorified}. The first level of 
categorification would associate to $\M$ some finite-dimensional {\em graded vector space} ${\mathcal H}^*(\M)$ so that 
\begin{equation}\label{eq_euler_char}\chi_\vir(\M) = \sum_j (-1)^j \dim {\mathcal H}^j(\M).\end{equation}
If $\M$ is nonsingular, with the Behrend function identically $(-1)^{\dim M}$, then we could just take for ${\mathcal H}^*(\M)$ 
the (shifted) classical cohomology $H^*(\M, \Q_\M[\dim \M])$. 

In the following sections, I~will explain the construction of the graded vector space ${\mathcal H}^*(\M)$ in general, 
starting from the moduli interpretation of $\M$. This vector space
will depend on more than just the scheme structure of $\M$; in modern 
language, we need to retain the derived moduli scheme structure (and a version of a symplectic form), not just the classical scheme
structure. We will use a classical truncation below, which will give this structure entirely in non-derived terms. 

A next level of categorification would involve associating a {\em category} ${\mathcal C}(\M)$ 
to the moduli space $\M$, some linearization of ${\mathcal C}(\M)$ being ${\mathcal H}^*(\M)$. 
I~will make some comments about this below (see in particular Remark~\ref{rem_cat_cat}); 
in general, it is not currently known how to build such a category.

\begin{remark} Closely related to categorified DT invariants are {\em refined DT invariants}. Here
the invariant attached to a moduli space $\M$ still takes values in a ring (and not a category), but it retains more 
information about $\M$ than just a number. 
Refined invariants in DT theory were conjectured to exist in the theoretical physics literature first~\cite{HIV, IK, IKV}.  
These works exploited the idea
that in theories built from local Calabi--Yau threefolds, there 
are certain extra symmetries that allow for a refinement (or quantization) of
the numerical invariant $\chi_\vir(\M)\in \Z$ to an invariant in $\Z[q, q^{-1}]$. 
In the above language, a reasonable
guess might be to associate to~$\M$ the expression $\sum_j q^j \dim {\mathcal H}^j(\M)$. As we will explain below 
(see Section~\ref{subsec:refined}), 
this turns out to be the wrong thing to do in general, for reasons of Hodge theory, but the correct answer is closely related. 

DT invariants can be defined in the most general ring where the cut-and-paste relation 
satisfied by an Euler characteris\-tic-type invariant holds, which is some version of the {\em Grothendieck ring of 
varieties}~$K({\rm Var}_\C)$. Invariants taking values in such rings are referred to as~{\em motivic DT invariants}, 
introduced in~\cite{KS1}. We will return to this point later. 
\end{remark} 

\section{A salient example}
\label{sec:hilbC3}

I~now introduce one salient example, in some sense the simplest infinite family of moduli 
spaces of sheaves on a Calabi--Yau threefold. Many of the features of the general theory appear here
in transparent form. The \cyt\ in question is simply affine three-space $Y=\C^3$. As we saw in the previous 
section, to be able to talk about moduli spaces, we need to assume some compactness of the objects we 
study; so in this example, we will be parametrizing collections of points of affine three-space. In other words, we will 
work in the context of Example~\ref{ex_rank1}, parametrizing 
rank-one torsion-free sheaves on $Y=\C^3$ with trivial determinant and compactly 
supported quotient. 

\subsection{The Hilbert scheme of points}

There are different ways of parametrizing collections of (unlabelled) points of a variety $Y$. The simplest way is 
to consider the symmetric product
\[ S(Y) = \bigsqcup_{n=1}^\infty S^n(Y) 
\]
where $S^n$ is the $n$-th symmetric product
\[ S^n(Y) = Y^n / \Sigma_n,
\]
the quotient by the symmetric group $\Sigma_n$ of the $n$-fold Cartesian product of $Y$. A point in $S(Y)$ 
corresponds to a collection of points of $Y$ with multiplicities. However, $S(Y)$ is not a moduli space: its points do not
parametrize objects on $Y$ which vary in a flat deformation family. From this point of view, 
better behaved is the {\em Hilbert scheme (of points)} of $Y$
\[ \Hilb(Y) = \bigsqcup_{n=1}^\infty \Hilb^n(Y).
\]
Here as a set, $\Hilb^n(Y)$, the $n$-th Hilbert scheme (of points) of $Y$, is given by
\[\Hilb^n(Y) = \{ \I \lhd \O_Y \colon \dim H^0(\O_Y/\I) = n\}.
\]
$\Hilb(Y)$ parametrizes flat families of finite-codimension ideals of the structure sheaf $\mathcal O_Y$ 
with fixed codimension; by a theorem of Grothendieck, its connected components $\Hilb^n(Y)$ are
(represented by) a separated scheme, which are quasiprojective if $Y$ is. 
It is well known that $\Hilb(Y)$ is nonsingular if $\dim Y\leq 2$; the first case when singularities
appear on the Hilbert scheme is when $\dim Y=3$. Starting in this dimension, the singularities are
really severe for large number of points $n$; in particular the $\Hilb^n(Y)$ are reducible, and not known to be reduced. 

\subsection{The Hilbert scheme of points of $\C^d$}

Let us work out in detail what the definition of $\Hilb^n(Y)$ says for the specific case of affine space $Y=\C^d$. 
I~am using set-theoretic notation, but it is easy to check that all the steps can be performed in families. 
First, since $\C^d$ is affine, we can forget about sheaves and think about rings instead.
\[\begin{array}{rcl}\Hilb^n(\C^d) & = & \{ \I \lhd \O_{\C^d} \colon \dim H^0(\O_{\C^d}/\I) = n\}\\
& = & \{ I \lhd \C[\O_{\C^d}] \colon \dim \C[\O_{\C^d}]/I = n\}\\
& = &\{ I \lhd \C[x_1,\ldots,x_d] \colon \dim \C[x_1,\ldots,x_d]/I = n\}.
\end{array}\]
Here $\C[\O_{\C^d}]$ is the ring of functions on affine $d$-space, and in the last line 
a set of affine coordinates were chosen. We can turn things round now: think of the $n$-dimensional 
vector space $V=\C[x_1,\ldots,x_d]/I$ as the primary object, which comes with a $\C[x_1,\ldots,x_d]$-module structure as
extra data. This can be encoded in operators $X_1,\ldots,X_d\in\GL(V)$, which are required to commute. Also, 
$\C[x_1, \ldots, x_d]/I$ is not an arbitrary module, but it is {\em cyclic}: it is generated as 
a $\C[x_1, \ldots, x_d]$-module by the image $v\in V$ of $1\in \C[x_1, \ldots, x_d]$. This vector $v$ 
is often called a {\em cyclic} or {\em framing vector}. 

Fixing an $n$-dimensional vector space~$V$, 
we can consider 
\[  M_n = \{ (X_1, \ldots, X_d, v)\colon [X_i, X_j]=0, \langle X_1, \ldots, X_d\rangle v = V\} \subset \GL(V)^d\times V,
\]
where $[\ ,\ ]$ denotes the standard commutator on matrices. Since the choice of $(X_i, v)$ is only determined 
up to automorphisms of $V$, we obtain a set-theoretic description 
\[  \Hilb^n(\C^d) = M_n/\GL(V).
\]

This concrete description turns out to be very useful; let us analyze it. Separate the data into two parts: 
there is the tuple $(X_1, \ldots, X_d, v)$ together with the 
cyclicity condition; and the relations $[X_i, X_j]=0$. The first part gives an auxiliary space
\[ N_n = \{ (X_1, \ldots, X_d, v)\colon \langle X_1, \ldots, X_d\rangle v = V\} \subset \GL(V)^d\times V, 
\]
the parameter space of cyclic modules for the free {\em non-commutative} $\C$-algebra generated by non-commuting 
variables $X_1, \ldots, X_d$, and its quotient
\[  \NCHilb^n_d = N_n/\GL(V).
\]
\begin{proposition}\label{prop:NCHilb}
The noncommutative Hilbert scheme $\NCHilb^n_d$ is a smooth quasi\-projec\-tive variety. The (commutative) Hilbert scheme $\Hilb^n(\C^d)$ is 
the subscheme of $\NCHilb^n_d$ cut out by the relations $[X_i, X_j]=0$.
\end{proposition} 
\begin{proof} Quasiprojectivity of $\NCHilb^n_d$ follows from the fact that $N_n/\GL(V)$ is a GIT quotient of the space of $d$-tuples of 
matrices under the action of $\GL(V)$; 
cyclicity corresponds to a particularly pleasant choice of stability condition~\cite[Sect. 1]{MS}. It is easy to see that cyclic 
triples are all stable with no non-trivial automorphisms, hence $\NCHilb^n_d$ is also nonsingular.
Finally, the commutation relations cut out the Hilbert scheme as a subscheme inside $\NCHilb^n_d$, in other words 
the scheme of commuting framed endomorphisms represents the functor of families of finite codimension
ideals in $\C[x_1, \ldots, x_d]$; the proof is given in the proof of~\cite[Thm. 1.9]{Nak}. 
\end{proof}

The space $\NCHilb^n_d$ is often called the {\em non-commutative Hilbert scheme}, the moduli space of codimension-$n$ ideals 
in the non-commutative ring $\C\langle X_1, \ldots, X_d\rangle$.  
For $d=2$, Proposition~\ref{prop:NCHilb} gives the well known description of the Hilbert scheme of the affine plane as the framed space of 
pairs of commuting matrices~\cite[Thm. 1.9]{Nak}.

\subsection{The Hilbert scheme of points of $\C^3$}

Let us now restrict to the case $d=3$, corresponding to the local Calabi--Yau threefold $Y=\C^3$. 
Consider the function $\widetilde W_n  \colon  N_n  \to  \C$ given by
\[
(X_1, X_2, X_3, v) \mapsto  \Tr X_1[X_2, X_3]. 
\]
From properties of $\Tr$ it follows that this map is more symmetric than it looks, being cyclically symmetric in the indices.
Now let us think about what is the degeneracy locus $\Crit(\widetilde W_n)$ of this function, the locus on $N_n$ given by the 
vanishing of the one-form ${\rm d}\widetilde W_n$. Forgetting the vector $v$ which does not enter into the function $\widetilde W_n$, 
local coordinates on $N_n$ are just the matrix entries of the $X_i$. In matrix entries, we have 
\[\widetilde W_n(X_1, X_2, X_3) = \sum_{i,j,k}
(X_1)_{ij}(X_2)_{jk}(X_3)_{ki} - \sum_{i,j,k}
(X_1)_{ij}(X_3)_{jk}(X_2)_{ki} .\]
Differentiating with respect to a matrix entry $(X_1)_{ij}$, the corresponding component of the one-form  ${\rm d}\widetilde W_n$ is precisely 
the $(j,i)$ entry of the commutator $[X_2, X_3]$. By cyclicity, this also holds for 
all the other differentiations. The vanishing of the one-form therefore precisely recovers the commutator relations, and
so the subset (subscheme) $\Crit(\widetilde W_n)$ of $N_n$ is exactly $M_n$.
Finally, $\tilde W_n$ is $\GL(V)$-invariant, so it descends to a function \[W_n\colon\NCHilb^n_3\to\C\] 
with degeneracy locus given by the quotient of the degeneracy locus on $N_n$, which is just $\Hilb^n(\C^3)$. 
We can summarize the discussion in the following result. 

\begin{proposition}\label{prop:hilb} The smooth noncommutative Hilbert scheme $\NCHilb^n_3$  carries the function
$W_n\colon \NCHilb^n_3\to\C$. The commutative Hilbert scheme $\Hilb^n(\C^3)$ is the scheme-theoretic critical locus
$\Crit(W_n)$ of the function $W_n$ on $\NCHilb^n_3$, the subscheme given by the equations $\{{\rm d}W_n=0\}$.
\end{proposition} 

\begin{remark} Instead of considering framed triples, we can also look at the space of triples of commuting matrices
\[M^0_n = \{ (X_1, X_2, X_3)\colon [X_i, X_j]=0\} \subset \GL(V)^3,\] 
and its quotient $M^0_n/\GL(V)$. Since the action is no longer very pleasant, it is best to think of this as the stack quotient
$[M^0_n/\GL(V)]$. An analogous argument to the above gives that this is isomorphic to the stack $\Tors^n(\C^3)$ of torsion sheaves
of length $n$ on $\C^3$, yet another way to parametrize $n$~points on $\C^3$. This again is a degeneracy locus, 
now in the stacky sense. 
\end{remark}

The significance of these results is that, as I~will discuss below, being a (scheme-theoretic) critical locus is exactly the right
local structure for a moduli space of sheaves on a Calabi--Yau threefold; this in turn will lead us to cohomological DT theory. 

\subsection{Numerical DT invariants} \label{sec:macmahon}
While the method of computation is not strictly relevant to the direction we are going to
take, for completeness I~mention the well known result for DT invariants. Let
\[ Z_{\C^3}(t) = 1 + \sum_{n\geq 1} \chi_\vir\left(\Hilb^n(\C^3)\right) t^n.
\]
\begin{theorem}\label{thm:macmahon} \cite{MNOP, BF} We have
\[Z_{\C^3}(t) = \prod_{m \geq 1}\left(1-(-t)^n\right)^{-n}.
\]
\end{theorem} 

One proof proceeds by exploiting the torus action on $\C^3$, which extends to each $\Hilb^n(\C^3)$ and has isolated fixed
points parametrized by three-dimensional partitions of $n$. It is a non-trivial result of \cite{BF} that each fixed point contributes
$(-1)^n$, leading to the above formula via combinatorics. 
The result can also be deduced from Theorem~\ref{thm:BBS} below. 

\section{Quivers with potential and their representations}
\label{sec:quiv}

Motivated by the previous example, I~will now introduce a larger class of examples of DT-like moduli spaces, 
coming from algebra rather than geometry. We need to start with some non-commutative algebra. 

\subsection{Quivers with potential}
A {\em quiver} $Q=(V(Q), E(Q), t, h)$ consists of a finite set $V(Q)$ of {\em vertices}, and a finite set $E(Q)$ of {\em arrows} 
between them, with each arrow $a\in E(Q)$ having a tail $t(a)\in V(Q)$ and a head $h(a)\in V(Q)$. 
For now, no further conditions are assumed, so multiple arrows and arrows with head and tail at the same vertex are both 
allowed. We call arrows $a_1, a_2$ composable if $h(a_1) = t(a_2)$, and we write their composition as $a_2a_1$. A {\em path} 
is a composition $a_n\ldots a_1$ of composable arrows; there is then an obvious notion of composable paths also. 
A {\em loop} based at a vertex $v$ is a path such that $v=h(a_n)=t(a_1)$. 
We allow an empty loop $s_v$ at each vertex $v\in V(Q)$.

The path algebra $\C Q$ of $Q$ is defined to be the $\C$-vector space generated by all paths, with the product defined on
generators to be the composite of two paths if they are composable, otherwise zero. 
It has a semisimple subalgebra generated by the empty loops $s_v$ at the vertices, 
which are idempotents summing to an identity element $1\in \C Q$.

The path algebra itself, from our point of view, is too non-commutative; we wish to impose some relations. The specific relations
we have in mind arise as follows. Let $\C Q_{\rm cycl}$ be the vector space spanned by loops, with two loops considered equal if they 
can be rotated into each other by cyclic permutation. Let $W\in \C Q_{\rm cycl}$ be an element of this space, a finite linear 
combination of loops up to cyclic permutation; we call $W$ a {\em potential}. 

Given an arrow $a\in E(Q)$, we define formal partial differentiation 
\[ \partial_a \colon \C Q_{\rm cycl} \to \C Q\]
to be the linear map given on simple loops without repeating edges by 
\[ \partial_a(a_n\ldots a_1) = \left\{ \begin{array}{ll} a_{i-1}\ldots a_2a_1 a_n \ldots a_{i+1} & \mbox{if } a=a_i, \\ 0 & \mbox{otherwise.}\end{array}\right.
\]
We extend $\partial_a$ to $\C Q_{\rm cycl}$ using $\C$-linearity and the Leibnitz rule. 
We then define the {\em Jacobi algebra} corresponding 
to the pair $(Q,W)$ to be the quotient 
\[ J(Q,W) = \C Q / \langle \langle \partial_a W | \ a\in E(Q) \rangle\rangle.
\]

\subsection{Examples}\label{sec:exquiver}

\begin{example} \label{ex:threeloop}
Let $Q$ have one vertex and three loop edges $a_1, a_2, a_3$. Then $\C Q \cong \C\langle X_1, X_2, X_3\rangle$ is the
free non-commutative polynomial algebra on three variables. Let $W=a_1a_2a_3 - a_2 a_1 a_3$. Then it is easy to check that, 
in cyclic notation,
\[ \partial_{a_i} W = [a_{i+1}, a_{i+2}],
\]
and so $J(Q,W)\cong\C[x_1, x_2, x_3]$, the free {\em commutative} polynomial algebra on three variables. This example is clearly 
related to the discussion in Section~\ref{sec:hilbC3}; further connections will be revealed later.
\end{example}
\begin{example} \label{ex:mckay}
For a substantial generalization of this example, consider the following situation. Let $\Gamma<\SL(3,\C)$ be
a finite abelian group of order $r$ acting on $\C^3$ by linear maps of determinant $1$. Without loss of generality, we can assume that 
$\Gamma$ acts by diagonal matrices. Following McKay, attach a quiver to $\Gamma<\SL(3,\C)$ as follows. 
\begin{enumerate}
\item Let the vertex set $V(Q)$ of $Q$ be the set of irreducible representations $\{\rho_0, \ldots, \rho_{r-1}\}$ of $\Gamma$. 
\item The diagonal embedding $\Gamma<\SL(3,\C)$ corresponds to the existence of indices $d_1, d_2, d_3$ such that 
the action of $\Gamma$ on $\C^3$ is via \[\C^3 \cong V=\rho_{d_1} \oplus \rho_{d_2} \oplus \rho_{d_3}.\]
\item Set the numbers of arrows in $Q$ from $\rho_i$ from $\rho_j$ to be $\dim \Hom_{\Gamma}(\rho_i, \rho_j\otimes V)$. 
In other words,
the number of arrows from $\rho_i$ from $\rho_j$ counts the number of copies of $\rho_i$ in $\rho_j\otimes V$. 
\end{enumerate}

Note that we have
a decomposition of the set of arrows \[E(Q) = E(Q)_1 \sqcup E(Q)_2 \sqcup E(Q)_3\] according to which summand
of $V$ a particular arrow came from. The resulting oriented graph $Q$ is oriented three-regular, with each vertex being the head, 
and the tail, of exactly one arrow from each $E(Q)_{i}$.

To define a potential on the quiver $Q$, we take the potential of the previous example, and make it equivariant. In other
words, let $\{a_{j3}a_{j2}a_{j1}\}\in\C Q_{\rm cycl}$ be the set of all loops in $Q$ up to cyclic permutation such that  
$a_{ji}\in E(Q)_{i}$; let $\{b_{k3}b_{k1}b_{k2}\}\in\C Q_{\rm cycl}$ be the set of all loops in $Q$ up to cyclic 
permutation with 
$b_{ki}\in E(Q)_{i}$. Now set
\[ W = \sum_j a_{j3}a_{j2}a_{j1} - \sum_k b_{k3}b_{k1}b_{k2} \in \C Q_{\rm cycl}.
\]
\begin{proposition} The Jacobi algebra $J(Q,W)$ defined by $\Gamma<\SL(3,\C)$ is isomorphic to the {\em twisted coordinate ring}
$\C[V]\star\C[\Gamma]$, where elements of $\C[\Gamma]$ commute past polynomials in $\C[V]$ using the $\Gamma$-action on $V$. 
This is a non-commutative algebra which is finite as a module over its centre $Z(J(Q,W))\cong \C[V]^\Gamma$, the invariant ring.
\end{proposition}
\begin{remark} If $\Gamma$ is a nonabelian subgroup of $\SL(3,\C)$, 
the definition of the McKay quiver remains the same, and much of the discussion 
generalizes; the obtained Jacobi algebra will still be Morita equivalent, 
though not necessarily isomorphic, to the twisted coordinate ring~\cite[Theorem 4.4.6]{G}.
\end{remark}
\end{example}
\begin{example} \label{ex:conifold}
Let $Q$ have two vertices $V(Q)=\{1, 2\}$ and four arrows, with $a_1, a_2$ from $1$ to $2$, and $b_1, b_2$ from $2$
to $1$. Following~\cite{KW}, set $W= a_1b_1a_2b_2 - a_1b_2a_2b_1$. 
\begin{proposition} The Jacobi algebra $J(Q,W)$ is a non-commutative algebra
which is finite as a module over its centre 
\[Z(J(Q,W))\cong \C[x_{11},x_{12},x_{21},x_{22}]/\langle x_{11} x_{22}-x_{12} x_{21}\rangle,\] 
with $x_{ij} = a_ib_j + b_ja_i$.
\end{proposition}
\end{example}
The following common features of the examples above stand out. 
\begin{enumerate}
\item The non-commutative algebra $J(Q,W)$ is module-finite over its centre $Z(J(Q,W))$.
\item The centre $Z(J(Q,W))$ is isomorphic to the coordinate ring of an {\em affine three-dimensional Gorenstein singularity},
in other words a local singular Calabi--Yau threefold. In the first case, this threefold is $\C^3$; in the second case, it is 
the quotient $\C^3/\Gamma$; in the last case, the threefold ordinary double point 
singularity $\{x_{11} x_{22}-x_{12} x_{21}\}\subset\C^4$.
\item The algebra $J(Q,W)$ is a $3$-Calabi--Yau algebra~\cite{G}, in particular satisfying the conditions (a)-(b) of 
Section~\ref{sec:3CYalg} (2).
\end{enumerate}
If properties (1)-(3) hold, then $J(Q,W)$ is called a {\em non-commutative crepant resolution}~\cite{vdB} of the singular
threefold $\bar Y = \mathrm{Spec}\, Z(J(Q,W))$. 
Properties (1)-(3) hold more generally in case $(Q,W)$ is constructed from a toric Gorenstein threefold 
singularity~$\bar Y$, as proved in~\cite[Thm. 8.6]{Br}. Given a general $(Q,W)$, some of the properties can fail. 
For further details of aspects of the general theory, see for example~\cite{Bock, MR, D1}. 

\begin{remark} While quiver algebras have been studied extensively in mathematics, the precise
construction explained in this section originated in supersymmetric physics, in studies aiming to understand
string theory on D-branes in Calabi--Yau threefolds in the language of
gauge theory. While the literature is too vast for us to survey here, important 
papers include~\cite{DM,DGJT, BD, FHetal, HV}. The latter papers develop, in the toric case, a dual combinatorial construction called 
the dimer model~\cite{Br}. The combinatorics of dimer models plays a role in computing numerical DT invariants 
on the corresponding toric models~\cite{Sz1, MR}. 
\end{remark}

\subsection{Representations of quivers with potential}\label{subsec:reps}

Let us continue to work with a general $(Q,W)$, whether or not properties (1)-(3) of Section~\ref{sec:exquiver} hold. 
Given a {\em dimension vector } $\dd=(d_i)\in \N^{V(Q)}$, a {\em representation 
of the quiver $Q$} is a collection of vector spaces $\{U_i, i\in V(Q)\}$ attached to the vertices
with $\dim U_i = d_i$, and linear maps $\{\phi_a\colon U_{t(a)}\to U_{h(a)}, a\in E(Q)\}$ attached to the arrows. 
There is an obvious notion of morphisms, and thus isomorphisms, of representations of~$Q$. 
Fixing $\dd\in \N^{V(Q)}$ and the set $\{U_i\}$, the space of all possible representations is 
\[ N_\dd = \prod_{a\in E(Q)} \Hom(U_{t(a)},  U_{h(a)}),\]
which is acted on by the group 
\[
G_\dd = \prod_{i\in V(Q)} \GL(U_i).
\]
The quotient stack $\cR_\dd(W)= [N_\dd/G_\dd]$ parametrizes isomorphism classes of representations of $Q$ of 
dimension vector $\dd$. 
We can think of such representations alternatively as (finite-dimensional) representations of the algebra $\C Q$, and it 
is easy to see that all finite-dimensional representations arise in this way.

It is not difficult to put the potential $W$ into the picture. Because $W$ is a linear combination of loops up to cyclic rotation,
its trace $\Tr W$ can be unambiguously evaluated on any representation.
The partial derivatives $\partial_a W$ consist of a linear combination of paths, and hence they can also be evaluated 
on any representation of $Q$. Just as before, requiring these combinations to vanish gives the degeneracy locus of
the function $\Tr W$. On the other hand, such representations are precisely the 
representations of the Jacobi algebra $J(Q,W)$. So set
\[ M_\dd = \{ (\phi_e) \in N_\dd\colon \partial_a(W)(\phi_e)=0 \mbox{ for all } a\in E(Q) \}\subset N_\dd.
\]
The quotient stack $\cR_\dd(W, Q)= [M_\dd/G_\dd]$ then parametrizes isomorphism classes of representations of $J(Q,W)$ of 
dimension vector $\dd$. 

Finally, if one prefers spaces to stacks, one way to go is to fix a framing vertex $j\in V(Q)$ and consider representations
with a cyclic vector $v\in U_j$, meaning that the images of $v$ under all possible 
combinations of $\phi_a$ span $\oplus_i U_i$. Thus set 
\[ N^j_{\dd} = \{ ((\phi_a), v) \colon v \mbox { is a cyclic vector for } (\phi_a) \} \subset \prod_{a\in E(Q)} \Hom(U_{t(a)},  U_{h(a)})\times U_j\]
to be the space of linear maps with choice of framing vector, and let $M^j_{\dd}$ be the subspace of 
representations satisfying the relations $\partial_a(W)(\phi_e)=0$ for all $a\in E(Q)$. 

\begin{proposition} \label{prop:quiverrepcrit}
\begin{enumerate}
\item The stack $\cR_\dd(Q)= [N_\dd/G_\dd]$ is a non-singular stack, equipped with a regular function 
$W_\dd=\Tr W \colon \cR_\dd(W)\to \C$ given by evaluating the trace of $W$ in a given representation. Its stack-theoretic 
critical locus is exactly $\cR_\dd(W, Q)$. 
\item Given a framing vertex $j\in V(Q)$, the quotient $\cR^j_\dd(Q)= N^j_\dd/G_\dd$
is a non-singular quasiprojective variety, equipped with a regular function $W^j_\dd=\Tr W\colon \cR^j_\dd(Q)\to \C$.
The scheme-theoretic critical locus of $W_\dd^j$ is the subscheme $\cR^j_\dd(W,Q)=M_\dd^j/G_\dd$ of 
$\cR^j_\dd(Q)$.
\end{enumerate}
\end{proposition}

\begin{remark} Choosing a framing vertex and considering cyclic representations is a special case of a more general 
construction, that of choosing a stability parameter and considering stable representations. For quivers,
the right notion of stability is King stability~\cite{K}. Fixing a dimension vector $\dd$, there is 
a space of (King) stability parameters 
\[ H_\dd = \{ \theta\in\Z^{V_Q} \colon \theta \cdot \dd = 0 \} \subset \Z^{V(Q)}. \] 
Choosing a generic $\theta\in H_\dd$ we get the moduli spaces $\cR^\theta_\dd(Q)\supset\cR^\theta_\dd(Q,W)$ 
of $\theta$-stable representations of $\C Q$, respectively $J(Q,W)$; both are quasi-projective GIT quotients,  
the former non-singular. 
Since we will not need the details of this construction, we will not recall them and refer instead to~\cite{K}.  
The framing case is when the $\theta$-weight of the vertex $j$ is chosen to be large positive, and all the other weights
sufficiently negative. 
\end{remark} 

We thus obtain a large supply of moduli spaces with the property that they arise as critical loci
of regular functions on smooth schemes or stacks. In the examples of the previous section, these spaces 
are sometimes familiar objects. 
\begin{enumerate}
\item For the three-loop quiver of Example~\ref{ex:threeloop} and dimension vector $\dd=(n)$, as the framed moduli
space we recover the Hilbert scheme of~$n$ points on~$\C^3$ from Section~\ref{sec:hilbC3}. The stack of unframed 
representations is the stack of torsion sheaves. 
\item For the McKay quiver of Example~\ref{ex:mckay}, from a finite abelian subgroup 
$\Gamma\in\SL(3,\C)$, choose the framing vertex to be at the trivial representation~$\rho_0$. Then for dimension vector 
$\dd=(1,\ldots, 1)$, the space $\cR^{\rho_0}_{\dd}(Q,W)$ is Nakamura's $\Gamma$-Hilbert scheme~\cite{Naka}. More general 
dimension vectors lead to equivariant (or stacky) Hilbert schemes of the quotient stack $[\C^3/\Gamma]$. Keeping
$\dd=(1,\ldots, 1)$ but changing the stability condition leads to various moduli spaces of $\Gamma$-constellations~\cite{CI}. 
\item Finally in the conifold case of Example~\ref{ex:conifold}, for 
certain specific stability parameters $\theta$ the spaces $\cR^\theta_\dd(W, Q)$ can be 
identified~\cite{NN} with geometric DT and PT (pairs) moduli spaces~\cite{MNOP, PT} of the resolved conifold geometry
$Y=\O_{\P^1}(-1,-1)$. 
\end{enumerate}

\section{Local and global structure of the moduli space}

Let us turn to the question in what generality the structure found in Section~\ref{sec:hilbC3} and 
generalized algebraically in Section~\ref{sec:quiv} can be expected to hold in geometric cases. Fix a projective
Calabi--Yau threefold $Y$, and let $\M$ be a moduli space of simple coherent sheaves on $Y$.

\subsection{Local structure} It is now known that, locally, the space $\M$ is a critical locus. 
Here is the precise result. 

\begin{theorem}\label{thm:localstr} Let $\M$ be a moduli space of simple coherent sheaves on the projective 
Calabi--Yau threefold~$Y$. Let $p\in\M$ be a closed point. Then there exists a quadruple $(R,U,f,i)$ with the following
properties: 
\begin{itemize}
\item $p\in R$ is a Zariski open neighbourhood of $p$ in $\M$; 
\item $f\colon U\to \C$ is a regular function on a smooth scheme $U$; 
\item $i\colon R \to U$ is an embedding, such that $i(R) = \Crit(f)$ as subschemes of $U$.
\end{itemize}
\end{theorem}

This is a hard theorem. In a gauge-theoretic, infinite-dimensional setup, 
it has been known and exploited for some time that the moduli space of holomorphic bundles on a Calabi--Yau threefold
can be described by the critical points of the Chern--Simons
functional on the space of connections~\cite{W, DT}.
While it is possible to cut this result down to a finite-dimensional model~\cite{JS}, the result one obtains
is necessarily weaker (non-algebraic). A finite-dimensional algebraic result which was weaker in a different way
has been known for some time, due to Behrend~\cite{Be}: the moduli
space is the zero-locus of a so-called {\em almost closed one-form}. The full statement above
follows from combining two difficult results. The first one states that the moduli scheme $\M$ has the structure of
a {\em smooth derived scheme with a $(-1)$-shifted symplectic structure}~\cite{PTVV}; 
I~refer to the original source for the definition of all these terms. The second result
is that a smooth $(-1)$-shifted symplectic derived scheme is locally of the stated form~\cite{BBJ, BG}. 

Theorem~\ref{thm:localstr} continues to hold if instead of simple coherent sheaves, we consider simple 
complexes, objects of the category $\D^b(\Coh Y)$. There is also an extension to stacks, for which we refer to~\cite{BBBJ}. 
Two variants of this result that would be useful, which however do not appear to be in the literature, is the existence
of a $(-1)$-shifted symplectic structure on $\M$ in case 
\begin{itemize} 
\item $Y$ is only quasiprojective, but we restrict the objects
parametrized so that $\M$ is still of finite type;
\item $\M$ is a moduli of objects in a 3-Calabi--Yau category $\cC$, perhaps with some extra properties. 
\end{itemize}
See~\cite[Thm.7.3.2]{Bus} for a partial result in the first case. 

It appears to be a difficult problem to find explicit local models
$(R,U,f,i)$ for specific geometric situations.

\begin{example} A set of examples well studied in the literature~\cite{Ka2, AK} where explicit potentials can be written down
is the case of (ideal sheaves of) rational curves $C\cong \P^1$ in
(local) Calabi--Yau threefolds $Y$. Here is one example 
from~\cite[Sect. 5.3]{AK}. Start with the Calabi--Yau threefold
\[ Y = \C^3_{x,y_1, y_2} \cup \C^3_{w,z_1,z_2}\]
glued from affine charts using the gluing map
\[\left\{ \begin{array}{rcl}w & = &x^{-1} \\
z_1 & = & x^3 y_1 + y_2^2 \\
z_2 & = & x^{-1}y_2
\end{array}\right\}.\]
This threefold contains the curve $C=\A^1_x\cup \A^1_w\cong\P^1$, with usual gluing map $w=x^{-1}$, 
given in the affine charts
by $y_i=0$, respectively $z_i=0$. The normal bundle is $\cN_C=\O_{\P^1}(-3, 1)$, as seen from the first terms
on the right hand side. A general section of $\cN_C$, giving a deformation of the curve, is given
in the first affine chart by the equations $y_1 = 0$, $y_2 = a + b x$, where $a, b$ are coordinates
on the two-dimensional first order tangent space. The deformation
space $R$ of the curve $C$ is given by $R=\Crit(f)$ for $f(a,b)=a b^2$ on $U=\C^2_{a,b}$. It is in fact easy  
to write down the whole one-dimensional family of curves corresponding to the reduced part $\A^1_a$ of $R$; 
just consider the map
\begin{eqnarray}\label{eq:laufercurve}\phi\colon \A^1_a \times \P^1& \to & Y = \C^3_{x,y_1, y_2} \cup \C^3_{w,z_1,z_2} \\
\nonumber (a, x) & \mapsto & (x,0,a)\\
\nonumber (a,w) & \mapsto & (w, a^2, aw).
\end{eqnarray}
For $a=0$, we have our original $(-3,1)$-curve $C_0=C$, whereas for $a\neq 0$, we have a $(-2,0)$-curve 
$C_a$ moving locally in a smooth one-parameter family.

If one changes the gluing data of $Y$ to \[\left\{ \begin{array}{rcl}w & = &x^{-1} \\
z_1 & = & x^3 y_1 + y_2^2+x^2 y_2^{2n+1} \\
z_2 & = & x^{-1}y_2
\end{array}\right\},\]
then one gets the famous Laufer example~\cite{L} of a 
non-deforming, contractible $(1,-3)$-curve $C\subset Y$, whose deformation space $R$ is the non-reduced, 
isolated critical locus $\Crit(f)$ of the function $f$ on $U=\C^2_{a,b}$ given by $f(a,b)=ab^2 + \frac{a^{2n+2}}{2n+2}$.
\label{ex:Laufercurve}
\end{example} 

\subsection{Global structure}

We call a quadruple $(R,U,f,i)$ as in Theorem~\ref{thm:localstr} 
a {\em critical chart} at $p\in\M$. We thus know that the moduli space $\M$ is covered by critical charts. The question
is what is the right ``minimal'' global structure which allows us to choose a compatible set of critical charts. 

This problem was solved by Joyce in~\cite{J}, where he introduced the notion of a {\em d-critical locus}, 
a pair $(X,s)$, where $X$ is a scheme, and $s$ is the section of a certain sheaf of vector spaces on $X$.
A d-critical locus is, locally in the Zariski topology, covered by critical charts $(R_j,U_j,f_j,i_j)$. The section $s$ 
essentially remembers $f_j$ up to second order on its critical locus $i_j(R_j)$, and serves to glue the various
charts consistently. I~refer to~\cite{J} for further details.
A d-critical locus has a well-defined {\em d-canonical line bundle} $K_{X,s}$, living on the reduced scheme $X^{\rm red}$, 
with the property that for every critical chart $(R_j,U_j,f_j,i_j)$, there is a natural isomorphism
\[K_{X,s}|_{R_{j,\red}}\cong i_j^*(K_{U_j}^{\otimes 2})|_{R_{j,\red}}.
\]
These isomorphisms are local, and are consequences of the fact that
$X$ is cut out of in each local $U_j$ by a section of $\Omega_{U_j}$. An {\em orientation} of the d-critical locus $(X,s)$ is a global choice of square root 
line bundle $K^{1/2}_{X,s}$ of the d-canonical line bundle $K_{X,s}$. We then have
\begin{theorem} Let $\M$ be a moduli space of simple coherent sheaves on a projective 
Calabi--Yau threefold $Y$. Then $\M$ obtains a d-critical locus structure from the moduli interpretation, 
which can be oriented.
\label{thm:globalcrit}\end{theorem}
The existence of the d-critical structure follows from~\cite{BBJ}, whereas the existence of square roots of the canonical
bundle is proved in~\cite{Hua, NO}. 

\begin{example} Joyce~\cite[Ex.2.39]{J} gives an example of a d-critical locus which 
cannot arise as a moduli space of simple sheaves or objects on a projective Calabi--Yau threefold, since it does not
admit an orientation. Let $X$ be the non-reduced scheme with $X_\red=\P^1$ containing a single embedded point $p\in\P^1$.
Let $f\colon\C^2\to\C$ be given by $f(x,y)=x^2y$. This defines a critical chart $(\Crit(f),\C^2,f,i)$ on $X$ near~$p$. Away from
$p$, we can take the trivial critical chart $(\A^1, \A^1, 0, j)$, with $j\colon \A^1\cong \P^1\setminus\{p\}$. We obtain
a d-critical locus structure $(X,s)$ on $X$ with $K_{X,s}|_{X_{\red}}\cong\O_{\P^1}(-5)$, hence non-orientable. 

Note that the critical chart $(\Crit(f),\C^2,f,i)$ itself was shown to be a moduli of curves and thus 
sheaves on a Calabi--Yau threefold $Y$ in Example~\ref{ex:Laufercurve}. In this example, it must be the case that in the compact
moduli space $\M$ of the ideal sheaf of the curve, there is also singular behaviour over the compactifying point.
It can in fact be argued 
directly from the geometry that this must be the case: the family~\eqref{eq:laufercurve} of curves cannot be completed with a smooth 
$(0,-2)$-curve over $\infty\in\M_\red\cong\P^1$.
\end{example}

An extension of Theorem~\ref{thm:globalcrit} to stacks can be found in~\cite{J,BBBJ}. 

\section{The DT sheaf and its cohomology}
\label{sec:DT:sheaf}

Suppose that $\M$ is an oriented d-critical locus, for example coming from a moduli problem on a Calabi--Yau threefold.
As explained in Section~\ref{sec:catDT}, 
we are looking for a cohomology theory ${\mathcal H}^*(\M)$ on $\M$, whose Euler characteristic is the DT invariant of $\M$ computed
from the Behrend function as in~\eqref{eq_euler_char}.

\subsection{Local construction} Suppose first that we are in the ideal situation of 
Sections~\ref{sec:hilbC3}-\ref{sec:quiv}, when the moduli scheme $\M$ is covered by a single critical 
chart $(R=\M,U,f,i)$. In this case, a candidate for a coefficient system on $\M$,
and thus the relevant cohomology theory, can be read off from Behrend's paper~\cite{Be}. Namely, he proves
that in this case the value of the Behrend function $\nu_\M(p)$ at a point $p\in \M$ 
agrees with the (signed) Euler characteristic of the reduced
Milnor fibre of the function $f\colon U\to\C$ at $p\in \Crit(f)=\M$. 
This in turn is known to agree with the Euler characteristic, 
with supports at $p$, of the {\em sheaf of vanishing cycles} of $f$, see Section~\ref{subsec:vanishing}. 
We can take this perverse sheaf \[\Phi=\phi_f\Q_U[\dim U]\] as our
coefficient system, the {\em DT sheaf}, on $\M$, and consider its 
(hyper)cohomology $H^*(\M, \Phi)$ as our {\em cohomological Donaldson--Thomas invariant}. For now, this is a graded vector 
space; it also carries the action of the monodromy endomorphism $T$. There is also a duality isomorphism
$\Phi\cong \DD_\M \Phi$ of $\Phi$ with its Verdier dual. 

\begin{remark} A pair $(U,f)$ of a smooth variety and a regular function gives rise to a plethora of cohomology
theories, both algebraic and topological; see for example~\cite[Sections 4 and 7]{KS2}. The cohomology 
of the sheaf of vanishing cycles, called {\em critical cohomology} by Kontsevich and Soibelman, is but one
possibility. Its main advantage over all other definitions is that is is defined ``on shell'', in terms of data
that can be described as living on the critical locus $\Crit(f)$ only. All other variants seem to require global
data of the pair $(U,f)$. Of course the study of pairs $(U,f)$ has received a lot of recent attention in the form 
of the study of Landau--Ginzburg models; the literature is too vast for us to survey it here. 
\end{remark}

\begin{example} For a harmless example, assume that $\M$ is nonsingular (in particular reduced). We can simply take 
the critical chart $(R=\M, U=\M, f=0, i={\rm id_\M})$. Then $\Phi = \Q_\M[\dim\M]$ is just the constant perverse sheaf, and 
the cohomological DT invariant is just the (shifted) classical cohomology of $\M$. 
\end{example}

\begin{example} Let $U=\C^2$, $f(x,y)=xy$. Then $R=\Crit(f)=p\subset\C^2$ is the reduced origin, and it can be checked
easily that $\Phi = \Q_p$. This example generalizes to split quadratic forms, which always give the trivial 
coefficient system on the (reduced) critical locus. This then becomes
the same example as the previous one, but in a different embedding. We
will get rid of the dependence on the embedding in the next section. 
\end{example} 

\begin{example} \label{ex:threeA1s}
Here is a small example with singular critical locus. Let $U=\C^3$ and $f(x,y,z)=xyz$. 
Then $R=\Crit(f)=\C_x\cup\C_y\cup\C_z$ is the
union of the three coordinate axes meeting at a point $p\in R$, the origin in $\C^3$. 
The vanishing cycle sheaf $\Phi=\phi_f\Q_{\C^3}[3]$, an object in the category $\Perv(R)$, must have simple factors 
which are intersection cohomology sheaves of strata in $R$  (see Section~\ref{subsec:const}). An easy computation shows that it has a 
composition series with simple factors $\Q_p$ (the trivial sheaf at
the singular point), the intersection cohomology sheaf $\IC_R$ of the
equidimensional $R$, and once again $\Q_p$. 
\end{example}

\begin{example} \label{ex:pfaff}
For a second non-trivial singular example, historically one of the first to be computed in this 
context, let $U=\C^{15}$, thought of as the space of $6\times 6$ skew-symmetric matrices. Let $f\colon U\to\C$ be given 
by the Pfaffian of the 
corresponding matrix. Then $R=\Crit(f)$ is reduced and irreducible, with a single singular point $p\in R$. It is 
easy to say more: the Pfaffian is a cubic function, and its partial derivatives with respect to the matrix entries are
the Pl\"ucker relations defining, after projectivisation, the subvariety $\Gr(2,6)\subset\P^{14}=\P U$. Thus $R$ is 
the affine cone over the Grassmannian, with its singular point $p\in R$ the cone point. 

The space $R$, by virtue of its construction, carries an oriented d-critical locus structure. As proved in~\cite{DS}, this 
structure arises from a genuine DT problem: up to a constant factor $\C^3$, it describes the singular locus of the 
Hilbert scheme $\Hilb^4(\C^3)$, the first time the Hilbert scheme is singular. The singular points of $\Hilb^4(\C^3)$
are exactly at the colength-$4$ ideals ${\bf m}_q^2$, the squares of maximal ideals of points $q\in\C^3$. 
This locus is a copy of $\C^3$, and the local structure is exactly given by the cone over the Grassmannian~\cite{Ka}. 

The DT sheaf $\Phi$ can also be analyzed in very concrete terms. It turns out to be non-semisimple, once again with a 
composition series with simple factors $\Q_p$ (the trivial sheaf at
the singular point), the intersection cohomology sheaf $\IC_R$, and $\Q_p$.
\end{example}

Let us finally discuss the simplest stacky situation, which will be important later. Let 
$\M=[M/G]$ be a global quotient stack, with $G$ reductive, arising as a stacky critical locus 
$\M= \Crit(f)\subset [N/G]$ where $f\colon N\to\C$ is a regular, $G$-invariant function. In this case, 
the vanishing cycle sheaf $\Phi=\phi_f \Q_N[\dim N]$ becomes an equivariant object~\cite{GKM} on $M$, and 
we can take its equivariant cohomology $H^*_G(M, \Phi)$~\cite{KS2} as our definition of 
the cohomological DT invariant of $\M$. It is naturally a module over $H^*_G({\rm pt}, \Q)$. 

\begin{remark}\label{rem_cat_cat} As explained in Section~\ref{sec:catDT}, the next level of categorification would involve
building a category ${\mathcal C}(\M)$ out of the moduli space $\M$. When $\M$ is covered by a single critical chart
$(R=\M,U,f,i)$, the relevant category is likely to be some version of
the {\em category of matrix factorizations}~\cite{O, LP}
${\rm MF}(U,f)$ of $f$ on $U$. The periodic cyclic homology of this
(dg) category is known to be~\cite{Efi} isomorphic to the cohomology
of the vanishing cycle sheaf of $f$. 
\end{remark} 

\subsection{Global issues} Suppose that $X$ is a scheme equipped with a d-critical locus structure, 
covered by a collection of critical charts
$\{(R_j,U_j,f_j,i_j): j\in J\}$. Locally, we can take the perverse sheaves $\Phi_j = \phi_{f_j}\Q_{U_j}[\dim U_j]$; 
the question is whether the d-critical structure carries enough information to glue these local $\Phi_j$ to a global 
sheaf $\Phi$ on $X$. We are certainly helped by the fact that perverse sheaves are known to glue uniquely, once compatible 
glueing isomorphisms are specified. As it turns out~\cite[Thm. 6.9]{BBJS}, this almost suffices: the d-critical structure
allows one to specify glueing isomorphisms {\em up to a sign}; an oriented d-critical structure gives
a unique glueing $\Phi$ on $X$. 

Given a moduli space $\M$ of simple sheaves on a projective Calabi--Yau threefold $Y$, by Theorem~\ref{thm:globalcrit}
it carries a d-critical structure and an orientation (a square root of its canonical bundle). 
Choosing an orientation, we can glue the local vanishing cycle sheaves to a global {\em DT sheaf} 
$\Phi$ on $\M$; it is important to remember that $\Phi$ depends on this further choice. Given all that though, we finally
have a definition of the {\em cohomological DT invariant of
$\M$}~\cite{BBJS, KL} as the (hyper)cohomology $H^*(\M, \Phi)$. 
This is a graded vector space, equipped with the action of an endomorphism $T$ glued from the local monodromy endomorphisms. 
The local duality isomorphisms also glue to an isomorphism $\Phi\cong \DD_\M \Phi$ of $\Phi$ with its Verdier dual. 

The case when $\M$ is a stack is covered again in~\cite{BBBJ}. The results contained in~\cite{BBBJ} 
allow one to define cohomological DT invariants for (finite type) moduli Artin stacks of sheaves (or complexes without negative
Ext groups) on Calabi--Yau threefolds. Note that in the presence of semistables, such a stack is not isomorphic to the moduli 
{\em scheme} of semistables, so the connection of this construction to results on moduli spaces containing semistables,
obtained by wall crossing~\cite{JS}, is not immediately clear. 

\begin{example} \label{ex:ruledsurface} Let $Y$ be a projective \cy\ threefold containing a 
divisor $E$ which is itself a $\P^1$-bundle $\pi: E\to C$ over a nonsingular projective curve $C$ of genus $g$. Assume also that 
$X$ admits a contraction to a singular projective threefold $\bar Y$ which contracts $E$ along $\pi$ and is an isomorphism outside $E$. 
Let $\beta\in H_2(Y,\Z)$ be the class of the fibre of $\pi$. Then for the lowest possible value of the Euler characteristic $n$, 
the moduli space of ideal sheaves $\M(Y, \beta,n)$ is isomorphic to $C$ itself, each point of $C$ parametrizing the ideal 
sheaf of the corresponding fibre of $\pi$. 

This is a non-singular (reduced) moduli space, and one d-critical
locus structure on it is given by the choice $(C, C, 0, {\rm id}_C)$, 
with d-canonical divisor $K_C^{\otimes 2}$. This admits the obvious square root $K_C$, giving the trivial coefficient
system $\Q_C[1]$ on $C$, the cohomological DT invariant
$H^*(C,\Q_C[1])$, and the numerical DT invariant $2g-2$. Note that for
$g=0$, this numerical DT invariant is negative. 

There are other choices of orientation: for any $2$-torsion line bundle $L\in\Pic^0(C)$, we get another 
orientation $K_C\otimes L$. A nontrivial $L$ corresponds to an unramified double cover $q:D\to C$. There is correspondingly 
a decomposition $q_*\Q_D = \Q_C \oplus \cL$, with $\cL$ a nontrivial rank-one local system on the curve $C$. With this choice
of orientation data, the cohomological DT invariant becomes the cohomology $H^*(C, \cL[1])$, a kind of Prym cohomology. 
\end{example}

\subsection{Hodge structure on cohomological DT and its purity}

As explained in the Appendix, given an algebraic variety or scheme
$X$, there is a category whose objects carry richer information than perverse sheaves on $X$: 
the category of mixed Hodge modules. In the case of a single critical
chart $(R, U, f, i)$, we can consider the mixed Hodge module 
\[  \phi_f\Q_U(\dim U/2)[\dim U]\in\MMHM(R)_\half\]
to be our Hodge-theoretic DT sheaf; see Section~\ref{subsec:Hodge} for
the definition of the category $\MMHM(R)_\half$ of {\em monodromic}
mixed Hodge modules on $R$ and its
half-twist functor, which is needed in the definition above to get duality right. 
Its cohomology carries a (polarized) mixed Hodge structure with monodromy action. 
Given a d-critical locus structure 
on a scheme $\M$, the glueing procedure described above, leading to the DT sheaf $\Phi$, can be performed on the level of 
mixed Hodge modules with monodromy~\cite{BBJS}. Consequently the cohomological DT invariant acquires a mixed Hodge structure 
too; the cohomology groups $H^i(\M, \Phi)$ and dually $H_c^i(\M, \Phi)$  become objects in the category $\MMHS_\half$ 
of monodromic mixed Hodge structures. 

\begin{example}\label{ex:isolated}
Let $f\colon \C^n\to \C$ be a polynomial, with an isolated critical point at $p\in \C^n$ which we will assume 
without loss of generality to be at the origin. By shrinking $\C^n$ to an open set $U$, we can assume that for $f\in U\to\C$, we
have $\Crit(f)_{\rm red} = \{p\}$ the origin. In this case, the cohomology $H^*(\phi_f\Q_U[\dim U])$ is essentially the cohomology 
of the Milnor fibre of $f$ at $p$. This cohomology carries a natural Hodge structure by work of Steenbrink~\cite{Steen}. In general,
this Hodge structure is mixed and not pure: it carries a non-trivial weight filtration. By another result of Steenbrink~\cite{Steen2}
however, the mixed Hodge structure is pure if~$f$ is {\em quasi-homogeneous}, that is if there is a set of integral weights on the 
coordinates of $\C^n$ which makes $f$ homogeneous of positive weight.
\end{example}

A natural question in the theory is what would in general guarantee the purity of the cohomological
DT invariant. Recall that in the classical case of constant coefficients, the cohomology $H^k(X, \Q)$ of a variety $X$ carries 
a mixed Hodge structure which is pure of weight $k$ if~$X$ is smooth and projective. As mentioned above, DT-style moduli spaces $\M$ are
smooth in the appropriate (derived) context, and so it would have perhaps been reasonable to expect that as soon as they are projective 
also, the cohomological DT invariant carries a pure Hodge structure. As Example~\ref{ex:isolated} shows, this already fails
when the reduced scheme underlying $\M$ is a single point. 
However, the final result quoted in the example also shows that one can hope for 
purity in the presence of a $\C^*$-action. It is also known in the literature on Landau--Ginzburg models
that the presence of suitable $\C^*$-actions leads to stronger results than are possible in
the general case. In this context, the following is the best result known. 

\smallskip

\begin{proposition} \cite{DMSS} Let $f\colon U\to\C$ be a regular function on a smooth variety~$U$. 
Assume that~$U$ carries a $\C^*$-action, so that~$f$ is equivariant with respect to the 
weight~$d$ action of~$\C^*$ on~$\C$, with $d>0$. Assume also that the critical locus $R=\Crit(f)$ of~$f$ is proper. 
Then the mixed Hodge structure on $H^*(R, \phi_f\Q_U)$ is pure. 
\label{prop:purity}\end{proposition}

Note that, as before, purity has to be interpreted appropriately in the monodromic sense.
This is not a difficult result given Saito's technology~\cite{S1, S2}; it is also far from optimal. It seems reasonable
to expect that if $\M$ is a projective scheme with a $\C^*$-action and a $\C^*$-equivariant oriented d-critical locus structure, then 
the cohomological DT invariant attached to $\M$ carries a pure Hodge structure. 

\begin{example} \label{ex:threeP1s} To give a non-trivial example of Proposition~\ref{prop:purity}, consider a compactification
of Example~\ref{ex:threeA1s}: let $X=\P^1\cup_p\P^1\cup_p\P^1$ be the
union of three projective lines meeting at a point $p\in R$. This scheme can be given a d-critical locus structure by using
$R$ from Example~\ref{ex:threeA1s} as one of the charts, and three trivial charts in the neighbourhoods of compactifying points. 
In turns out however~\cite[Sect. 6]{Ef} that $X$ can even be covered by a single critical chart $(R',U',f',i')$, using an 
auxiliary quiver 
construction. For this chart, all conditions of Proposition~\ref{prop:purity} are satisfied. Indeed, $H^*(X, \Phi)$
can be computed to be isomorphic to the cohomology of two disjoint copies of $\P^1$, and thus indeed pure.
\end{example}

\begin{example}\label{ex:coni} Another interesting example arises in the stable pairs theory of the local resolved
conifold geometry $Y=\O_{\P^1}(-1,-1)$, discussed first in~\cite[Section 4.1]{PT}. For the lowest degree curve class,
all moduli spaces are nonsingular. For the case of the curve class having multiplicity two, for the lowest value of the 
Euler characteristic the moduli space is still nonsingular. The next moduli space $\M$ is more interesting. 
The reduced variety underlying~$\M$ is isomorphic to $\P^3$. As discussed in~\cite[Section 4.1]{PT} however, this cannot be the 
full answer: the numerical pairs invariant equals $4$ and not $-4$, which would be the answer in case 
the moduli space were just a smooth $\P^3$. The moduli space $\M$ is in fact a non-reduced scheme, 
the thickening of $\P^3$ along the embedded quadric $Q\subset\P^3$, with Zariski tangent spaces of 
dimension $4$ along the quadric. 

The corresponding DT sheaf was computed, under an assumption about the local form of the potential, in~\cite[Ex.4.5]{Sz}. 
Indeed, the simplest potential that gives the quoted non-reduced behaviour is of the form $f=x^2y$ in some local variables
$x,y$ on a smooth space $U$. Assuming that this is indeed the local form of the potential at all points of the embedded 
quadric in $\P^3$, it can be shown that the DT sheaf $\Phi$ on $\M$ is the intersection cohomology sheaf $\IC(\cL)$ for a rank-one local 
system $\cL$ on $\P^3\setminus Q$ with nontrivial $\Z/2$ monodromy. 

On the other hand, it follows from~\cite{NN} that this space $\M$ is also a global degeneracy locus, and for the
corresponding critical chart all conditions of Proposition~\ref{prop:purity} are satisfied. Indeed,  $H^*(\M, \Phi)$
is isomorphic to the cohomology of the quadric $Q$, and thus pure. 
\end{example}

\subsection{Summary of properties of the cohomological DT invariant}

In summary, let $\M$ be a scheme equipped with an oriented d-critical locus structure; let
$\Phi$ be the corresponding DT (perverse) sheaf on $\M$. Then the cohomological DT invariant
$H^*(\M, \Phi)$ attached to $\M$ is a graded vector space with the following properties and carrying the following structures. 
\begin{enumerate}
\item The Euler characteristic of the cohomological DT invariant 
\[ \sum_i (-1)^i \dim H^i(\M, \Phi) = \int_\M \nu_\M d\chi
\]
agrees with the numerical DT invariant computed from the Behrend function $\nu_\M$ of $\M$. 
\item $H^*(\M, \Phi)$ carries the action of the monodromy endomorphism $T$.
\item The self-duality of $\Phi$ under Verdier duality implies that
  there is a duality isomorphism between $H^*(\M, \Phi)$ and
  $H_c^*(\M, \Phi)$. In particular, if $\M$ is proper, then $H^*(\M, \Phi)$ is self-dual (satisfies Poincar\'e duality). 
\item $H^*(\M, \Phi)$ carries a monodromic mixed Hodge structure. 
\item If $\M$ is proper, and covered by a single $\C^*$-equivariant critical chart with positive $\C^*$-weight on the base, 
then the monodromic mixed Hodge structure on $H^*(\M, \Phi)$ is pure. 
\end{enumerate}
This in particular applies when $\M$ is a projective moduli scheme of stable
sheaves on a smooth projective Calabi--Yau threefold $Y$. The cohomological, as opposed
to the numerical, DT invariant is generally {\em not} invariant under
deformations of the underlying Calabi--Yau threefold $Y$; see Section~\ref{subsec:def_inv}
below for further discussion. 

\section{Some computational results}

\subsection{The cohomological DT invariant in the $K$-group}
\label{subsec:refined}
It appears difficult to compute cohomological DT invariants directly. Most existing computations use an indirect route, 
computing on an appropriate locally closed stratification of the moduli space. Such results will give weaker answers than the
computation of the full cohomology of the DT sheaf, since cohomology is not additive on stratifications. To be able to
use the long exact sequence on cohomology, we also switch to compactly
supported cohomology (which, by Poincar\'e duality, 
carries the same information). So given $(\M, \Phi)$, we set
\[ [H_c^*(\M, \Phi)] = \sum_i (-1)^i[H^i_c(\M, \Phi)] \in K(\MMHS_\half)
\] 
to be the formal alternating sum of cohomologies inside the $K$-group of monodromic mixed Hodge structures. 
For this invariant, given a decomposition
$\M = U \cup Z$ into an open subscheme and its complement, the long exact sequence in cohomology indeed gives
\begin{equation} [H_c^*(\M, \Phi)] = [H_c^*(U, \Phi|_U)] + [H_c^*(Z, \Phi|_Z)]\in K(\MMHS_\half).
\label{eq:Kth}\end{equation}

Because of cancellations in the $K$-group, an element $\alpha\in K(\MMHS_\half)$ does not allow one to reconstruct the original 
complex $A^*$ of monodromic mixed Hodge structures giving rise to it, or even the dimensions of the various pieces of cohomology. 
One piece of information that survives is the {\em weight polynomial}, defined in Section~\ref{sec_Kmixed}, living in $\Z[q^{\pm\half}]$.
The corresponding one-variable deformation or quantization of the numerical DT invariant is often referred
to as a {\em refined DT invariant}~\cite{IKV, DG}.

On the other hand, the situation changes completely if we know that the element $\alpha=[H_c^*(\M, \Phi)]$ 
comes from $H_c^*(\M, \Phi)$ with {\em pure} Hodge structure. In that case, there can be no cancellation between 
different terms in~\eqref{eq:Kth}, and so the knowledge of the $K$-theory element enables one to fully determine the dimensions
of the cohomology groups $H_c^*(\M, \Phi)$, and hence $H^*(\M, \Phi)$. In this case, the weight polynomial becomes simply the Poincar\'e 
polynomial.

In the next two subsections, we give the answers in the $K$-group to some cohomological DT problems from earlier, 
and discuss what is known about the full cohomology.

\subsection{The cohomological DT invariants of $\C^3$}

Consider 
\[
Z_{\C^3}(t) = 1 + \sum_{n\geq 1} [H_c^*(\Hilb^n(\C^3), \Phi_n)] t^n \in K(\MMHS_\half)[[t]],
\]
with $\Phi_n$ the DT sheaf on $\Hilb^n(\C^3)$, obtained from the critical locus interpretation of Proposition~\ref{prop:hilb}.
This series turns out to admit an elegant infinite product form in
terms of the cohomology $\L$ of the affine line (see
Section~\ref{sec_Kmixed}). 

\begin{theorem} \cite{BBS} We have
\[Z_{\C^3}(t) = \prod _{m=1}^{\infty }\prod _{k=0}^{m-1}\left(1-\L^{k+2-\frac{m}{2}}\,t^{m} \right)^{-1} \in K(\MMHS_\half)[[t]].
\]
\label{thm:BBS}
\end{theorem}

This is the precise analogue of G\"ottsche's formula, which computes the cohomology of the Hilbert scheme of points 
of a smooth quasiprojective surface, for the space $Y=\C^3$. It also allows one to confirm a computation of refined DT invariants 
in the physics literature~\cite{IKV}. Setting $\L^{\half}=-1$, we
recover Theorem~\ref{thm:macmahon}. It was very recently proved in
\cite{Dav2} that the critical cohomology carries 
a pure Hodge structure, so this formula encodes all the  cohomological DT invariants in this example.

\begin{example} For the case of four points, an expansion gives
\[ [H_c^*(\Hilb^4(\C^3), \Phi_4)] = \L^6 + \L^5 + 3 \L^3 + 3 \L^3 + 3 \L^2 + \L +1 \in K(\MMHS_\half).\]
This result was proved earlier in~\cite{DS}. 
\end{example}

\subsection{Other examples}

Other local cases where values of cohomological DT invariants are known, at least in the $K$-group, 
include~\cite{MMNS, MN, DaM}; the last reference gives a computation
which involves nontrivial monodromy on the DT invariant. 
I~spell out just one more example. Let $Y=\O_{\P^1}(-1,-1)$ be the resolved conifold geometry, containing the
zero-section $Y\supset C\cong\P^1$. Let $\cN_{d[C],n}$ be the moduli space of stable pairs~\cite{PT}, supported on the curve class $d[C]$
and with Euler characteristic $n$, carrying the DT sheaf $\Phi_{d,n}$. Let
\[ Z_Y(t,T) = \sum_{n,d} \left[H^*\left(\cN_{d[C],n}, \Phi_{d,n}\right)\right] T^d t^n.
\]
\begin{theorem}\cite{MMNS} We have
\[ Z_Y(t,T) = \prod_{m\geq 1}\prod_{j=0}^{m-1} \left(1- \L^{-\frac m2+\half+j} t^m T \right).
\]
\end{theorem}
In this case, it is known from the quiver description of these moduli spaces~\cite{NN} that Proposition~\ref{prop:purity} applies, and 
the underlying cohomology is pure; therefore this result gives a full computation of the cohomology. The cohomological 
DT invariant of Example~\ref{ex:coni} can be read off from expanding this series.

\section{The Kontsevich--Soibelman cohomological Hall algebra}

One of the reasons to be interested in a cohomological extension of Donaldson--Thomas theory is that since this
invariant takes values in a category rather than a ring, one can
define maps between different cohomology groups. Indeed,
starting from a quiver with potential, Kontsevich--Soibelman~\cite{KS2} showed how one can define an associative
algebra structure on cohomological DT. This construction can be viewed as a precise mathematical realization of the 
Harvey--Moore algebra of BPS states~\cite{HM}. 

\subsection{Quivers without potential}

Let us first cover the case when we start with a quiver $Q$ with potential $W=0$. As in Section~\ref{subsec:reps}, 
given a dimension vector $\dd=(d_i)\in \N^{V(Q)}$, we fix vector spaces $\{U_i\}$ of the right dimension. 
The space of all representations 
\[ N_\dd = \prod_{a\in E(Q)} \Hom(U_{t(a)},  U_{h(a)}),\]
is acted on by the group 
\[
G_\dd = \prod_{i\in V(Q)} \GL(U_i).
\]
We get the quotient stack $\cR_\dd(W)= [N_\dd/G_\dd]$ and its cohomology $H^*_{G_\dd}(N_\dd, \Q)$. Let
\[ \cH_Q = \bigoplus_{\dd\in \N^{V(Q)}}H^*_{G_\dd}(N_\dd, \Q).
\]
\smallskip

\begin{theorem} \cite{KS2} The vector space $\cH_\Q$ admits a product operation, under which it becomes an associative
unital $\N^{V(Q)}$-graded algebra.
\end{theorem}

We refer to the original paper for the definition of the product, which essentially arises from the standard correspondence
involving subrepresentations and quotient representations, familiar from Hall algebra theory. Several examples are
computed in~\cite{KS2}. For example, it is clear that when $Q$ has only one vertex, $\cH_Q$ is isomorphic as a vector space 
to $\oplus_{n>0}H^*(\BGL(n, \C)))$. The algebra structure depends on the number of loops; for the case of no loop, respectively one 
loop, the algebra $\cH_Q$ is isomorphic~\cite[Sect. 2.5]{KS2} 
to the exterior, respectively symmetric algebra on infinitely many generators. 

A quiver $Q$ is called {\em symmetric} if for any $i,j\in V(Q)$ it has the same number of oriented edges from vertex $i$ to 
vertex $j$ as from vertex $j$ to vertex $i$. In the special case when $Q$ is symmetric, it is possible to use the cohomological
grading, together with a slight twist of the multiplication, to endow $\cH_Q$ with a $\N^{V(Q)}\times\Z$-graded supercommutative
structure. It was conjectured in~\cite{KS2}, and proved in~\cite{Ef1}, that in fact this algebra is {\em free} supercommutative. 

\begin{remark} It is important to note that the COHA itself only depends on a quiver (possibly with potential, see next section), not on any
notion of stability. Choosing a stability function gives a decomposition (involving a spectral sequence)  of the COHA into an
ordered product of more elementary algebras, similar to a decomposition for universal enveloping algebras which comes from 
the Poincar\'e--Birkhoff--Witt theorem. For details, see~\cite[Section 5]{KS2}.
\end{remark}

\subsection{Inserting the potential} \label{subsec_COHA_pot}
Let us equip $Q$ with a nonzero potential $W$. As in Proposition~\ref{prop:quiverrepcrit},
we get the function $W_\dd\colon N_\dd\to\C$, its critical locus $M_\dd=\Crit(W_\dd)$, and the stack 
$[M_\dd/G_\dd]$ of representations of the Jacobi algebra $J(Q,W)$. We set
\[\cH_{Q,W} = \bigoplus_{\dd\in \N^{V(Q)}}H^*_{G_\dd}(M_\dd, \phi_{W_\dd}\Q_{N_\dd}[\dim N_\dd-\dim G_\dd]).
\]
Once again, this becomes~\cite{KS2, Dav} a graded associative algebra. 

The structure of this algebra was studied further recently in~\cite{DavM2}, giving in particular an 
algebraic meaning to the BPS invariants $\Omega^\sigma(\alpha)$~of Section 1.5 in these quiver cases.
There are essentially no cases where the structure of
this algebra is fully known in concrete terms, though the dimensional reduction technique of~\cite{BBS}, adapted to the 
COHA in~\cite{KS2, Dav}, should allow interesting computations.

\subsection{Representations of the COHA}

One of the most exciting potential directions is to find representations of the COHA associated to a quiver $(Q,W)$, 
and related algebras, on the cohomological DT invariants attached to {\em stable} representations of $(Q,W)$ under appropriate 
stability conditions~\cite{Sz,S}. 

\begin{example} Let $(Q,W)$ be the three-loop quiver with potential from Example~\ref{ex:threeloop}. Then we have the cohomological 
Hall algebra $\cH_{Q,W}$, which should have an action on the vector space
\[ \cA = \bigoplus_{n\geq 0} H^*(\Hilb^n(\C^3), \Phi_n).
\]
Here $\Phi_n$ is the DT sheaf on $\Hilb^n(\C^3)$ obtained from the critical locus interpretation of Proposition~\ref{prop:hilb}. More 
precisely, there should be {\em two} actions, corresponding to adding and subtracting a representation, giving ``creation''
and ``annihilation'' operators. The resulting action of a double algebra $\widetilde\cH_{Q,W}$ would be the three-dimensional
analogue of the Grojnowski--Nakajima action~\cite{Gro, Nak1} of an infinite-dimensional Heisenberg algebra 
on $\bigoplus_{n\geq 0} H^*(\Hilb^n(\C^2), \Q)$.
Theorem~\ref{thm:BBS} could in turn be interpreted as a (refined) character formula for this action.
\end{example}

For the much simpler example of the $A_1$ quiver, the double COHA was explicitly computed recently in~\cite{X}.

\subsection{The global case} Constructing algebra structures on globally defined cohomological DT spaces,
involving moduli spaces covered by several critical charts, appears to be a very interesting 
challenge. Essentially nothing is known in this direction. 

\section{Further directions}

Without giving all, or indeed sometimes any, details, and without aiming for completeness, I~discuss some further research directions 
related to cohomological Donaldson--Thomas theory.

\subsection{Deformation (in)variance}\label{subsec:def_inv} As explained in Section~1, 
one of the properties of numerical DT invariants is that they are
constant in a smooth deformation family of projective Calabi--Yau
threefolds. One can ask more generally about deformation properties of
other flavours of DT invariants. For the case of the Hilbert scheme
of points of a general threefold, the result of~\cite{BBS} shows that
the motivic invariants only depend, additionally, on the motive of the
underlying threefold $Y$, as one would minimally expect. Thus, for
example, the weight polynomial refinement of DT remains constant in a
smooth deformation family, since it depends only on the weight
polynomial of $Y$, which is constant in a family. More
recently, \cite{KKP} observed (conjectural) constancy of the weight
polynomial realization of motivic stable pairs theory on thickened K3 surfaces. 

Refined DT invariants do not always need to be deformation
invariant however. 

\begin{example} \label{ex:ruledsurface:again} Let us return to the example studied
in Example~\ref{ex:ruledsurface}, with $Y$ a projective \cy\ threefold containing a 
divisor $E$ which is itself a $\P^1$-bundle $\pi: E\to C$ over a curve $C$ of genus $g$. 
It is well known that for $g>0$, the generic deformation $Y_t$ of $Y=Y_0$ contains $2g-2$ isolated rational curves in the
homology class of the fibre of the fibration $\pi$. Thus the moduli
space deforms from $\M_0=C$ to $\M_t$ being a disjoint union of
$2g-2$ reduced points. Clearly the numerical DT invariant $2g-2=
(-1)^{\dim C}\chi(C)$ is unchanged, but the weight polynomial, using
{\em trivial} orientation data, changes from the (shifted) Poincar\'e
polynomial of $C$ to the constant polynomial $2g-2$ and thus
jumps. Interestingly, it can be readily checked that with any choice
of {\em nontrivial} (Prym) orientation data, the weight polynomial
(though not the Hodge polynomial) remains {\em unchanged}.
\end{example}

In a different direction, a jump in refined DT invariants can also be
observed in a local situation, under homogeneous changes of potentials
on certain simple quivers~\cite{CMPS}.

\subsection{Knot invariants} Start with a singular curve $C$ with planar
singularities. One can then form the generating series of (ordinary,
topological) Euler characteristics of the Hilbert schemes of points
$\Hilb^n(C)$ of $C$. This series, and various refinements, were
conjectured by Oblomkov--(Rasmussen--)Shende~\cite{OS, ORS} to be given by formulae derived from the
HOMFLY(PT) invariants of the links around the different singular
points. These connections were later re-phrased and generalized~\cite{DSV, DHS}
via embedding $C$ in a local Calabi--Yau threefold,
and studying certain DT invariants of this threefold. Indeed, Maulik~\cite{M} recently proved the
original conjectures of~\cite{OS} using wall crossing in DT theory
(and other tools). Cohomological DT invariants in this story may be related
to categorifications of knot invariants such as Khovanov homology. 

\subsection{Quantum cluster positivity} Starting with a quiver $Q$ and
a fixed vertex $j\in V(Q)$, Fomin--Zelevinsky~\cite{FZ} introduced a
combinatorial operation called {\em mutation}, leading to another
quiver $\mu_j(Q)$. This is an operation that can then be repeated on
other vertices, usually leading to an infinite number of other quivers
mutation equivalent to $Q$. The mutation operation also associates
certain Laurent polynomials to vertices of all these quivers, related
by the so-called cluster transformation rule. It was a
conjecture of Fomin--Zelevinsky, recently proved in~\cite{SL}, that
all these Laurent polynomials have positive coefficients. 

The cluster mutation story has two extensions relevant for us. In one
direction, \cite{DWZ} proved that mutation extends under certain conditions
to an operation on quivers with potential. In~\cite{KY} it was shown that
this gives rise to very natural derived equivalences between
3-Calabi--Yau (dg) categories. Kontsevich--Soibelman~\cite{KS1} and
Nagao~\cite{Nag} then showed that wall
crossing of DT invariants of these categories reproduces a lot of the
structure of the cluster polynomials. 

In another direction, Berenstein--Zelevinsky~\cite{BZ} introduced a quantization of the
cluster story, introducing a version of the cluster transformation
rule in a mildly non-commutative, quantized Laurent polynomial ring. This allows one
to formulate a quantum cluster positivity conjecture, saying that even
in the quantized setting, all coefficients in the Laurent polynomials
are positive. 

The two lines of thought were brought together in work of Efimov~\cite{Ef},
who showed that certain very specific motivic DT invariants exactly
reproduce the cluster transformations of {\em quantized} Laurent
polynomials. As noticed in that paper, a purity result on {\em
  cohomological} DT invariants would give positivity for
quantum cluster polynomials.  By Proposition~\ref{prop:purity} above,
purity holds in certain $\C^*$-equivalent situations, leading to the
result~\cite{DMSS} that quantum cluster positivity is true for quivers
which can be equipped with a suitably non-degenerate {\em quasi-homogeneous} potential.
The proof of the quantum cluster positivity conjecture in complete
generality, relying still on cohomological Donaldson--Thomas theory, 
was very recently announced in~\cite{Dav3}.

\subsection{Localization} Since the seminal paper~\cite{GP}, many computations 
involving virtual fundamental classes were performed in the presence of a high degree of
(torus) symmetry, using localization. In the case of symmetric obstruction theories, 
localization is even more effective~\cite{BF}; as mentioned,
Theorem~\ref{thm:macmahon} was first proved
using torus localization. Away from the numerical case however, localization has so far 
proven much less effective. Many refined DT results, starting with~\cite{BBJ}, have used 
a dimensional reduction trick, which again relies on a certain kind of torus action; 
but this does not reduce the computation to torus-fixed points directly. Recently, there 
has been progress on this problem: Maulik~\cite{M2} introduced a framework 
which can, at least in theory, compute refined DT invariants via
localization in substantial examples, such as 
the case investigated earlier using a hands-on definition in~\cite{CKK}. It would
be interesting to generalize this work to cohomological DT invariants, where, 
as for intersection cohomology~\cite{Ki}, hyperbolic localization~\cite{Bra} is likely to play
an important role. 

\subsection{K-theoretic DT invariants} 

A radically new point of view on refinements of DT theory was presented recently by Nekrasov and 
Okounkov in~\cite{NO}. This work defines curve and sheaf counting on Calabi--Yau threefolds with values
in K-theory, taking the K-theory class of the {\em virtual structure sheaf} on the relevant moduli 
stack of sheaves. This construction is a holomorphic variant of the constructible DT-sheaf discussed 
in Section~\ref{sec:DT:sheaf}. 
Moreover, \cite{NO} also discusses a conjectural, not fully constructed theory of curve- (more precisely brane-)counting
on a Calabi--Yau fivefold $Z$, which connects to K-theoretic DT invariants on Calabi--Yau threefolds arising as 
torus-fixed loci inside the fivefold $Z$. Maulik's work~\cite{M2}, in turn, connects the K-theoretic and motivic
DT invariants under some (currently strong) hypotheses. The relationships between these various flavours of DT theory, 
as well as connections to geometric engineering, definitely deserve further study. 

\subsection{DT theory in higher dimensions}

Following on from one theme of the previous subsection, one may well wonder whether a sensible theory of 
sheaf counting exists on Calabi--Yau $d$-folds of arbitrary dimension greater than $3$. 
One hint that such a theory might exist
is that the product form of the generating series of motivic DT invariants of Hilbert schemes of points of a threefold $Y$, 
presented in~\cite{BBS}, admits a natural generalization for $\dim Y=d$ arbitrary, and gives the correct result both 
for $\dim Y<3$ and for small numbers of points. We thus have the answer for Hilbert schemes in general, we only need to 
formulate the right question! More importantly perhaps, the derived geometry results of~\cite{PTVV, BBJ, BG} are all 
dimension-independent: the general results are that moduli stacks of sheaves on a smooth $d$-dimensional projective Calabi--Yau 
variety carry a $(d-2)$-shifted symplectic structure, and there are normal forms for such structures generalizing the 
critical locus description for $d=3$. These results were used as the starting point to explore four-dimensional
DT theory in~\cite{BoJ}. A related approach relying on gauge theory was presented in~\cite{CCL}. 
It would be interesting to compute invariants of substantial examples in either of these frameworks. 

\appendix

\section{Perverse sheaves, vanishing cycles, and Hodge modules}

In this Appendix, I~recall aspects of the categories of constructible
and perverse sheaves and Hodge modules on complex algebraic varieties,
and the vanishing cycle construction, mainly to fix notation. Further
details can be found in~\cite{D, PS}; I~will not give original
references in this section. 

\subsection{Constructible and perverse sheaves}
\label{subsec:const}
Let~$X$ be a complex algebraic variety or scheme. The category ${\rm Sh}(X)$ of sheaves
of $\Q$-vector spaces in the classical (complex) topology has a subcategory $\Const(X)$ consisting of
{\em constructible} objects, sheaves which are local systems when
restricted to different strata of an {\em algebraic} stratification
of~$X$. Let $\D^b_c(X)$ denote the derived category of bounded complexes of
sheaves on $X$ whose cohomology sheaves are in $\Const(X)$.

If $f\colon X\to Y$ is a morphism of schemes, then we have pushforward
functors  $f_*, f_!\colon \D^b_c(X)\to \D^b_c(Y)$, the direct
image and direct image with compact supports functors.  There are also pullback
functors $f^*, f^!\colon \D^b_c(Y)\to \D^b_c(X)$, pullback and
extraordinary pullback. We also have a tensor product
$\stackrel{L}\otimes$, and a (Verdier) duality functor $\DD_X \colon
\D^b_c(X)\to \D^b_c(X)^{\rm op}$. These satisfy all the rules of the
six operation formalism; in particular, Verdier duality intertwines
between pushforward and compactly supported pushforward. 

For any scheme $X$ with structure morphism $a\colon X\to {\rm pt}$,
and the constant vector space $\Q\in \D^b_c( {\rm pt})$, let
$\Q_X=a^*\Q$. If $X$ is smooth, then we have an isomorphism
$\DD_X\Q_X\cong \Q_X[2\dim X]$. 
On the other hand, the cohomologies of $a_*\Q_X$, respectively
$a_!\Q_X$ are isomorphic to the cohomology $H^*(X)$ and compactly
supported cohomology  $H_c^*(X)$ of $X$ respectively. Thus, for
example, we get the classical fact that for smooth $X$, cohomology and
compactly supported cohomology are naturally dual. 

The derived category $\D^b_c(X)$ has an interesting and
nontrivial heart $\Perv(X)$, called the {\em category of perverse
  sheaves on $X$} even though its objects are not sheaves but
certain complexes defined by support conditions. This category is
preserved by Verdier duality, pushforward under closed inclusions,
(shifted) pullback by smooth maps, and tensor product. If $X$ is
smooth, then $\Q_X[\dim X]\in\Perv(X)$. Also $\Perv(X)$ is artinian and
noetherian, with every object having a finite filtration with
quotients being {\em intersection cohomology sheaves} $\IC_Y(\cL)$ 
where $Y\subset X$ is a smooth, locally closed, irreducible
subvariety, and $\cL$ is a local system on $Y$. 

Despite not being sheaves but complexes, perverse sheaves on complex schemes form a {\em stack}: 
locally defined perverse sheaves can be glued to a global object, given compatible isomorphisms. 

\subsection{The sheaf of vanishing cycles}
\label{subsec:vanishing}

Given a regular function $f\colon U\to\C$ with zero-locus $U_0=f^{-1}(0)$, there is a functor, the {\em vanishing cycle 
functor} $\phi_f\colon \D^b_c(U)\to \D^b_c(U_0)$, mapping $\Perv(U)$ to $\Perv(U_0)$. For any $\F\in\D^b_c(U)$
the image $\phi_f\F$ carries an extra piece
of structure, the monodromy endomorphism $T$. In the case of the constant perverse sheaf $\Q_U[\dim U]$ on a smooth 
variety $U$, the only case needed here, the vanishing cycle perverse sheaf can be viewed as an object 
in $\Perv(X)$ where $X=\Crit(f)$ is the critical locus of $f$. Since Verdier duality commutes with taking vanishing cycles,
the vanishing cycle sheaf of the constant sheaf is equipped with an isomorphism
\[\phi_f\Q_U[\dim U] \cong \DD_X \phi_f\Q_U[\dim U]\]
with its Verdier dual. 

\subsection{Hodge modules}
\label{subsec:Hodge}
Given a smooth complex variety $U$, there is an enriched category, the
category of {\em (polarized) mixed Hodge modules}~\cite{S1, S2}
$\MHM(U)$ with a faithful functor $\rat\colon\MHM(U)\to\Perv(U)$. With a very substantial amount of work, the six-functor
formalism extends to this category. For singular schemes $X$, the category of mixed Hodge modules $\MHM(X)$ needs
to be defined via embedding $X$ into a smooth variety. 

For $U={\rm pt}$, the category of mixed Hodge modules on $U$ is the category $\MHS$ of mixed Hodge structures. 
Thus the advantage of this generalization is that given $\F\in\MHM(X)$, its cohomology and cohomology with compact support, which 
coincide with its pushforward (with compact support) to the point, acquire mixed Hodge structures. 

A further complication is that while for $f\colon U\to\C$, the vanishing cycle functor $\phi_f$ 
extends to mixed Hodge modules, the image objects 
under this functor acquire the action of the monodromy endomorphism. Thus the image lies in a new category of 
{\em monodromic} mixed Hodge modules $\MMHM(U_0)$. One can either treat this category directly as a category of objects with
endomorphism~\cite{S3}, or one can treat it geometrically using an auxiliary disc~\cite{KS2}. Here I~will informally adopt the former
approach, 
so objects of $\MMHM(U)$ are simply pairs of objects $(\F, T)$ of $\MHM(U_0)$ and an endomorphism $T$. However, one needs
to be careful about the definition of various functors. For example, the monodromic Verdier duality functor
$\DD_X^M \colon \MMHM(U_0)\to \MMHM(U_0)^{\rm op}$ involves a certain shift. See~\cite[Sect. 2.10]{BBJS} for a slightly more
detailed but still concise treatment.

The vanishing cycle cohomology of the pair $(U=\C, f=x^2)$ has tensor
square isomorphic to the cohomology $\Q(-1)[-2]$ of the affine
line. This object itself can therefore be considered as a half-twist
object $\Q(-1/2)[-1]\in\MMHS$. One can formally extend the category
$\MMHS$ with a tensor inverse $\Q(1/2)[1]$ to this object; we denote
this new category by $\MMHS_\half$. By pullback, we get extended
categories of monodromic mixed Hodge modules $\MMHM(X)_\half$ on
general complex schemes~$X$.

\subsection{The $K$-group of mixed Hodge structures}\label{sec_Kmixed}
The $K$-group $K(\MMHS_\half)$ is a ring, though one has to be careful with the definition of multiplication in the presence
of monodromy. It contains a distinguished element $\L\in K(\MMHS_\half)$, the cohomology of the affine line $\C$. 
This element admits a square root $\L^\half$, the vanishing cycle
cohomology of the pair $(U=\C, f=x^2)$, which has a formal inverse $\L^{-\half}$ in $K(\MMHS_\half)$. 

There is a ring homomorphism $\chi\colon K(\MMHS_\half)\to\Z$ given by computing the Euler characteristic of
a complex of mixed Hodge structures. This maps $\L$ to $1$ and, to be consistent with conventions used elsewhere, 
$\L^\half$ to $-1$. There are other invariants however which survive. The Poincar\'e polynomial (with coefficients the dimension
of cohomology in different degrees) is not one of them. However, the {\em weight polynomial} or Poincar\'e--Serre polynomial 
gives a homomorphism $\wt\colon  K(\MMHS_\half)\to\Z[q^{\pm\half}]$. This is defined as follows: for an ordinary complex 
of mixed Hodge structures $A^*$, define
\[\wt(A^*;q) = \sum_{j\in \Z} \sum_{i\in \Z} (-1)^i \dim\left(Gr^j_W H^i(A^*)\right) \in \Z[q^{\pm \half}],
\]
with $Gr^j_W H^i(A^*)$ the weight-$j$ part of H$^i(A^*)$; 
let also $\wt(\L^{-\half})=q^{-\half}$. For a monodromic complex, define 
\[\wt(A^*;q) = \wt(A_1^*;q) +q\cdot \wt(A_{\neq 1}^*;q) \in \Z[q^{\pm \half}];
\]
here, $A_1^*$ is the part of the complex where the monodromy acts by eigenvalue $1$, and $A_{\neq 1}^*$ is a complement.
There is also a Hodge version of this polynomial, taking into account the Hodge filtration also,
which I~will not spell out in detail. The Euler characteristic can of course be recovered from the weight polynomial
by setting~$q^\half=-1$.


\begin{thebibliography}{99} 
\bibitem{AK} P. Aspinwall and S. Katz, Computation of superpotentials for D-branes. Comm. Math. Phys. 264 (2006) 227--253. 
\bibitem{BBBJ} O. Ben-Bassat, C. Brav, V. Bussi and D. Joyce, A `Darboux Theorem' for shifted symplectic structures on derived Artin stacks, with applications, arXiv:1312.0090.
\bibitem{Be} K. Behrend, {Donaldson--Thomas type invariants via microlocal geometry}, Ann. of Math. (2) 170 (2009) 1307--1338.
\bibitem{BBS} K. Behrend, J. Bryan and B. Szendr\H oi, Motivic degree zero Donaldson-Thomas invariants, Inv. Math. 192 (2013) 111--160.
\bibitem{BF1} K. Behrend and B. Fantechi,  The intrinsic normal cone, Invent. Math. 128 (1997) 45--88. 
\bibitem{BF} K. Behrend and B. Fantechi, Symmetric obstruction theories and Hilbert schemes of points on threefolds, Algebra Number Theory 2 (2008), 313--345. 
\bibitem{BZ}  A. Berenstein and A. Zelevinsky, Quantum cluster algebras, Adv. Math. 195 (2005) 405--455.
\bibitem{BD} D. Berenstein and M. Douglas, Seiberg duality for quiver gauge theories, arXiv:hep-th/0207027.
\bibitem{Bock} R. Bocklandt, Graded Calabi Yau algebras of dimension 3, J. Pure Appl. Algebra 212 (2008) 14--32. 
\bibitem{BoJ} D. Borisov and D. Joyce, Virtual fundamental classes for moduli spaces of sheaves of Calabi--Yau four-folds, manuscript. 
\bibitem{BG} J. Bouaziz and I. Grojnowski, A $d$-shifted Darboux theorem, arXiv:1309.2197. 
\bibitem{BBJ} C. Brav, V. Bussi and D. Joyce, A `Darboux theorem' for derived schemes with shifted symplectic structure, arXiv:1305.6302.
\bibitem{Bra} T. Braden, Hyperbolic localization of intersection cohomology, Transf. Groups  8 (2003) 209--216. 
\bibitem{BBJS} C. Brav, V. Bussi, D. Dupont, D. Joyce and B. Szendr\H oi, Symmetries and stabilization for sheaves of vanishing cycles, arXiv:1211.3259. 
\bibitem{B} T. Bridgeland, Stability conditions on triangulated categories, Ann. of Math. (2) 166 (2007) 317--345.
\bibitem{BKR} T. Bridgeland, A. King and M. Reid, The McKay correspondence as an equivalence of derived categories, J. Amer. Math. Soc. 14 (2001) 535--554. 
\bibitem{Br} N. Broomhead, Dimer models and Calabi--Yau algebras, Mem. Amer. Math. Soc. 215 (2012).
\bibitem{Bus} V. Bussi, Derived symplectic structures in generalized
  Donaldson--Thomas theory and categorification, D.Phil. Thesis,
  University of Oxford, 2014.
\bibitem{CMPS} A. Cazzaniga, A. Morrison, B. Pym and B. Szendr\H oi, work in progress. 
\bibitem{CKK} J. Choi, S. Katz and A. Klemm, {The refined BPS index from stable pair invariants}, arXiv:1210.4403v1.
\bibitem{CCL} Y. Cao and N. C. Leung, Donaldson-Thomas theory for Calabi--Yau 4-folds, arXiv:1407.7659.
\bibitem{CI} A. Craw and A. Ishii, Flops of $G-\Hilb$ and equivalences
  of derived categories by variation of GIT quotient, Duke Math. J.124
  (2004) 259--307.
\bibitem{D1} B. Davison, Consistency conditions for brane tilings, J. Algebra 338 (2011) 1--23.
\bibitem{Dav} B. Davison, The critical CoHA of a quiver with potential, arXiv:1311.7172v4.
\bibitem{Dav2} B. Davison, The integrality conjecture and the cohomology of preprojective stacks, arXiv:1602.02110.
\bibitem{Dav3} B. Davison, Positivity for quantum cluster algebras, arXiv:1601.07918.
\bibitem{DaM} B. Davison and S. Meinhardt, The motivic Donaldson-Thomas
  invariants of $(-2)$-curves, arXiv:1208.2462.
\bibitem{DavM2} B. Davison and S. Meinhardt, Donaldson-Thomas theory for categories of homological dimension one with potential, arXiv:1512.08898.
\bibitem{DMSS} B. Davison, D. Maulik, J. Schuermann and B. Szendr{\H{o}}i, Purity for graded potentials and cluster positivity, arXiv:1307.3379, to appear Comp. Math. 
\bibitem{DWZ} H. Derksen, J. Weyman and A. Zelevinsky, Quivers with potentials and their representations I: Mutations, Selecta Math. 14 (2008) 59--119.
\bibitem{DHS} D.-E. Diaconescu, Z. Hua and Y. Soibelman, HOMFLY polynomials, stable pairs and motivic Donaldson-Thomas invariants,
  Commun. Number Theory Phys. 6 (2012) 517--600. 
\bibitem{DSV} D.-E. Diaconescu, V. Shende and C. Vafa, Large $N$ duality, Lagrangian cycles, and algebraic knots,  Comm. Math. Phys. 319 (2013) 813--863. 
\bibitem{D} A. Dimca, Sheaves in topology, Universitext, Springer, 2004. 
\bibitem{DS} A. Dimca and B. Szendr{\H{o}}i, {The Milnor fibre of the Pfaffian and the Hilbert scheme of four points on {$\mathbb C^3$}},  Math. Res. Lett. 16 (2009) 1037--1055.
\bibitem{DG} T. Dimofte and S. Gukov, {Refined, motivic, and quantum}, Lett. Math. Phys. {91} (2010) 1--27. 
\bibitem{DT} S. Donaldson and R. Thomas, Gauge theory in higher dimensions, in: The geometric universe (Oxford, 1996), 31--47, Oxford Univ. Press, Oxford, 1998. 
\bibitem{DGJT} M. Douglas, S. Govindarajan, T. Jayaraman and A. Tomasiello, D-branes on Calabi--Yau manifolds and superpotentials, Comm. Math. Phys. 248 (2004) 85--118.
\bibitem{DM} M. Douglas and G. Moore, D-branes, quivers and ALE Instantons, arXiv:hep-th/9603167.
\bibitem{Ef1} A. Efimov, Cohomological Hall algebra of a symmetric quiver. Compos. Math. 148 (2012) 1133--1146.
\bibitem{Ef} A. Efimov, {Quantum cluster variables via vanishing cycles}, arXiv:1112.3601v2. 
\bibitem{Efi}  A. Efimov, Cyclic homology of categories of matrix factorizations, arXiv:1212.2859.
\bibitem{FZ} S. Fomin and A. Zelevinsky, Cluster algebras I: Foundations, J. Amer. Math. Soc. 15 (2002) 497--529.
\bibitem{FHetal} S. Franco, A. Hanany, K. D. Kennaway, D. Vegh, and B. Wecht, Brane dimers and quiver gauge theories,  J. High Energy Phys. 2006, no. 1, 096, 48 pp.
\bibitem{G} V. Ginzburg, Calabi--Yau algebras,  arXiv:math/0612139. 
\bibitem{GV} R. Gopakumar and C. Vafa, M-Theory and Topological Strings I, II, arXiv:hep-th/9809187, arXiv:hep-th/9812127.
\bibitem{GKM} M. Goresky, R. Kottwitz and R. MacPherson, Equivariant cohomology, Koszul duality, and the localization theorem, Invent. Math. 131 (1998) 25--83. 
\bibitem{GP} T. Graber and R. Pandharipande, Localization of virtual classes, Inv. Math. 135 (1999) 487--518.
\bibitem{Gro} I. Grojnowski, Instantons and affine algebras. I. The Hilbert scheme and vertex operators, Math. Res. Lett. 3 (1996) 275--291. 
\bibitem{HV} A. Hanany and D. Vegh, Quivers, Tilings, Branes and Rhombi, J. High Energy Phys. 2007, no. 10, 029. 
\bibitem{HM} J. Harvey and G. Moore, On the algebras of BPS states, Comm. Math. Phys. 197 (1998) 489--519.
\bibitem{HIV} T. Hollowood, A. Iqbal and C. Vafa, {Matrix models, geometric engineering and elliptic genera}, J. High Energy Phys. 03 (2008) 069. 
\bibitem{Hua} Z. Hua, Orientation data on moduli space of sheaves on Calabi--Yau threefold, arXiv:1212.3790.
\bibitem{IK} A. Iqbal and A.-K. Kashani-Poor, $\SU(N)$ geometries and topological string amplitudes, Adv. Theor. Math. Phys. {10} (2006) 1--32. 
\bibitem{IKV} A. Iqbal, Kozcaz and C. Vafa, {The refined topological vertex}, J. High Energy Phys. 10 (2009) 069. 
\bibitem{J1} D. Joyce, Configurations in abelian categories. IV. Invariants and changing stability conditions. Adv. Math. 217 (2008) 125--204. 
\bibitem{J} D. Joyce, A classical model for derived critical loci, arXiv:1304.4508, to appear J. Diff. Geom.
\bibitem{JS} D. Joyce and Y. Song, A theory of generalized Donaldson-Thomas invariants, Mem. Amer. Math. Soc. 217 (2012).
\bibitem{Ka}  S. Katz, The desingularization of $\Hilb^4(\P^3)$ and its Betti numbers, in: Zero-dimensional schemes (Ravello, 1992), 231--242, de Gruyter, Berlin, 1994.
\bibitem{Ka2} S. Katz, Versal deformations and superpotentials for rational curves in smooth threefolds, in: Symposium in Honor of C. H. Clemens (Salt Lake City, UT, 2000), 129--136, Contemp. Math., 312, Amer. Math. Soc., Providence, RI, 2002.
\bibitem{KKP} S. Katz, A. Klemm and R. Pandharipande, On the motivic stable pairs invariants of K3 surfaces, arXiv:1407.3181.
\bibitem{KY} B. Keller and D. Yang, Derived equivalences from mutations of quivers with potential, Adv. Math. 226 (2011) 2118--2168.
\bibitem{KL} Y.-H. Kiem and J. Li, Categorification of Donaldson-Thomas invariants via perverse sheaves, arXiv:1212.6444.
\bibitem{K} A. King, Moduli of representations of finite-dimensional algebras, Quart. J. Math. 45 (1994) 515--530. 
\bibitem{Ki} F. Kirwan, Intersection cohomology and torus actions, J. Am. Math. Soc. 1 (1988) 385--400. 
\bibitem{KW} I. R. Klebanov and E. Witten, {Superconformal field theory on threebranes at a Calabi--Yau singularity}, Nucl. Phys. B 536 (1998) 199--218.
\bibitem{KS1} M. Kontsevich and Y. Soibelman, {Stability structures, motivic Donaldson--Thomas invariants and cluster transformations}, 2008, arXiv:0811.2435.
\bibitem{KS2} M. Kontsevich and Y. Soibelman, {Cohomological Hall algebra, exponential Hodge structures and motivic Donaldson--Thomas invariants}, {Commun. Number Theory Phys.} 5 (2011) 231--352.
\bibitem{L} H. Laufer, On $\C\P^1$ as exceptional set, In: Recent Developments in Several Complex Variables (J. Fornaess, ed.), Ann. of Math. Stud. Vol. 100, Princeton Univ. Press, Princeton, NJ 1981, 261--275.
\bibitem{SL} K. Lee and R. Schiffler, Positivity for cluster algebras, arXiv:1306.2415.
\bibitem{LP} K. Lin and D. Pomerleano, Global matrix factorizations, Math. Res. Lett. 20 (2013) 91--106.
\bibitem{M} D. Maulik, Stable pairs and the HOMFLY polynomial,  arXiv:1210.6323. 
\bibitem{M2} D. Maulik, Motivic residues and Donaldson--Thomas theory, manuscript. 
\bibitem{MNOP} D. Maulik, N. Nekrasov, A. Okounkov and R. Pandharipande, {Gromov--Witten theory and Donaldson--Thomas theory. I}, Compos. Math. {142} (2006) 1263--1285.
\bibitem{MMNS} A. Morrison, S. Mozgovoy, K. Nagao and B. Szendr{\H{o}}i, {Motivic Donaldson--Thomas invariants of the conifold and the refined topological vertex}, Adv. Math. 230 (2012) 2065--2093.
\bibitem{MN} A. Morrison and K. Nagao, Motivic Donaldson-Thomas invariants of toric small crepant resolutions, arXiv:1110.5976.
\bibitem{MR} S. Mozgovoy and M. Reineke, On the noncommutative Donaldson-Thomas invariants arising from brane tilings, Adv. Math. 223 (2010) 1521--1544.
\bibitem{MS} D. Mumford and K. Suominen, Introduction to the theory of moduli, in: Algebraic geometry, Oslo 1970 (Proc. Fifth Nordic Summer-School in Math.), pp. 171--222, Wolters-Noordhoff, Groningen, 1972. 
\bibitem{Nag} K. Nagao, Donaldson-Thomas theory and cluster algebras, Duke Math. J. 162 (2013) 1313--1367. 
\bibitem{NN} K. Nagao and H. Nakajima, {Counting invariant of perverse coherent sheaves and its wall-crossing}, Int. Math. Res. Not. (2011) 3885--3938. 
\bibitem{Nak1} H. Nakajima, Heisenberg algebra and Hilbert schemes of points on projective surfaces. Ann. of Math. (2) 145 (1997) 379--388. 
\bibitem{Nak} H. Nakajima, Lectures on Hilbert schemes of points on surfaces, AMS University Lecture Series 18. 
\bibitem{Naka} I. Nakamura, Hilbert schemes of abelian group orbits, J. Algebraic Geom. 10 (2001) 757--779. 
\bibitem{NO} N. Nekrasov and A. Okounkov, Membranes and sheaves, arXiv:1404.2323.
\bibitem{O} D. Orlov, Triangulated categories of singularities and D-branes in Landau--Ginzburg models, Proc. Steklov Inst. Math. 2004 (246) 227--248.
\bibitem{OS} A. Oblomkov and V. Shende, The Hilbert scheme of a plane curve singularity and the HOMFLY polynomial of its link, Duke Math. J. 161 (2012) 1277--1303.
\bibitem{ORS} A. Oblomkov, J. Rasmussen and V. Shende, The Hilbert scheme of a plane curve singularity and the HOMFLY homology of its link, arXiv:1201.2115. 
\bibitem{PT} R.~Pandharipande and R.~P. Thomas, {Curve counting via stable pairs in the derived category}, Invent. Math. \textbf{178} (2009) 407--447.
\bibitem{PT_13} R.~Pandharipande and R.~P. Thomas, {13/2 ways of counting curves}, in: Moduli spaces (eds. L. Brambila-Paz, P. Newstead, O. Garcia-Prada, R. P. Thomas), CUP, 2014.
\bibitem{PTVV} T. Pantev, B. To\"en, M. Vaqui\'e and G. Vezzosi, Shifted symplectic structures, Publ. Math. Inst. Hautes \'Etudes Sci. 117 (2013), 271--328.
\bibitem{PS} C. Peters and J. Steenbrink, Mixed Hodge Structures, Ergebnisse der Mathematik und ihrer Grenzgebiete. 3. Folge, Vol. 52, Springer. 
\bibitem{S1} M. Saito, Modules de Hodge polarisables, Publ. RIMS 24 (1988) 849--995.
\bibitem{S2} M. Saito, Mixed Hodge Modules, Publ. RIMS 26 (1990) 221--333.
\bibitem{S3} M. Saito, Thom--Sebastiani theorem for mixed Hodge modules, preprint 1990/2011.
\bibitem{Steen} J. Steenbrink, Mixed Hodge structure on the vanishing cohomology, in: Real and complex singularities, 525--563, Sijthoff and Noordhoff, Alphen aan den Rijn, 1977.
\bibitem{Steen2}  J. Steenbrink, Intersection form for quasi-homogeneous singularities, Comp. Math. 34 (1977) 211--223.
\bibitem{S} Y. Soibelman, Remarks on cohomological Hall algebras and their representations, arXiv:1404.1606. 
\bibitem{Sz1} B. Szendr{\H{o}}i, {Non-commutative Donaldson--Thomas invariants and the conifold}, Geom. Topol. \textbf{12} (2008) 1171--1202.
\bibitem{Sz} B. Szendr{\H{o}}i, Nekrasov's partition function and refined Donaldson--Thomas theory: the rank one case, SIGMA (2012) 088.
\bibitem{T} R. Thomas,  A holomorphic Casson invariant for Calabi--Yau 3-folds, and bundles on K3 fibrations, Jour. Diff. Geom. 54 (2000) 367--438.
\bibitem{vdB} M. van den Bergh, Non-commutative crepant resolutions, in: The legacy of Niels Henrik Abel, 749--770, Springer, Berlin, 2004. 
\bibitem{W} E. Witten, Chern-Simons gauge theory as a string theory, in: The Floer memorial volume, 637--678, Progr. Math., 133, Birkh\"auser, Basel, 1995.
\bibitem{X} X. Xiao, The double of representations of Cohomological Hall Algebra for $A_1$-quiver, arXiv:1407.7593.
\end{thebibliography}
\end{document}